\documentclass[10pt]{amsart}

\usepackage{graphicx}

\usepackage{amssymb}
\usepackage{graphicx}

\usepackage{amssymb}
\usepackage{amsmath}
\usepackage{amsthm}
\usepackage{amscd}
\usepackage{url}
\usepackage{color}

\theoremstyle{plain}
\newtheorem{them}{Theorem}[section]
\newtheorem{pro}[them]{Proposition}
\newtheorem{cor}[them]{Corollary}
\newtheorem{lem}[them]{Lemma}

\theoremstyle{definition}
\newtheorem{defi}[them]{Definition}

\theoremstyle{remark}
\newtheorem{rem}[them]{Remark}
\newtheorem{ex}[them]{Example}

\DeclareMathOperator{\id}{Id} 
 
\DeclareMathOperator{\proj}{proj}

\DeclareMathOperator{\spt}{spt}

\DeclareMathOperator{\law}{Law}

\newcommand{\eps}{\varepsilon}
\newcommand{\Lg}{\lambda}
\renewcommand{\P}{\mathbb{P}}
\newcommand{\p}{\mathcal{P}}
\newcommand{\M}{\Pi_M} 
\newcommand{\m}{\mathcal{M}}
\newcommand{\R}{\mathbb{R}}

\newcommand{\N}{\mathbb{N}}

\newcommand{\V}{\mathcal{V}}
\newcommand{\diag}{\mathrm{Diag}}
\newcommand{\B}{\mathcal{B}}
\newcommand{\AND}{\quad\text{and}\quad}

\newcommand{\dd}{\mathrm{d}}

\newcommand{\E}{\mathbb{E}}

\newcommand{\sto}{\mathrm{sto}}

\newcommand{\leqp}{\preceq_+} 
\newcommand{\leqc}{\preceq_{C}} 
\newcommand{\leqs}{\preceq_\sto} 
\newcommand{\leqcp}{\preceq_{C,+}} 
\newcommand{\leqe}{\preceq_{C,+}} 
\newcommand{\leqcs}{\preceq_{C,\sto}} 
\newcommand{\leqps}{\preceq_{+,\sto}} 
\newcommand{\leqcps}{\preceq_{C,+,\sto}} %
\newcommand{\geqc}{\succeq_{C}} 
\newcommand{\Top}{\mathrm{Top}}
\newcommand{\down}{\mathrm{Down}}
\newcommand{\lc}{\mathrm{lc}}

\newcommand{\LIM}{\mathrm{LimCurt}}
\newcommand{\LFD}{\mathrm{LimCurtFD}}
\newcommand{\curt}{\mathrm{Curt}}
\newcommand{\e}{\mathrm{e}}

\renewcommand{\subset}{\subseteq}
\renewcommand{\supset}{\supseteq}

\title[Stability of shadows and curtain couplings.]{Stability of the shadow projection and the left-curtain coupling.}
\author{Nicolas JUILLET}
\address{Institut de Recherche Math\'ematique Avanc\'ee\\
UMR 7501, Universit\'e de Strasbourg et CNRS\\
7 rue Ren\'e Descartes\\
67000 Strasbourg, France}
\email{nicolas.juillet@math.unistra.fr}

\thanks{The author is partially supported by the ``Programme ANR ProbaGeo'' (ANR-09-BLAN-0364) and the ``Programme ANR JCJC GMT'' (ANR 2011 JS01 011 01).}

\begin{document}

\begin{abstract}

The (left-)curtain coupling, introduced by Beiglb\"ock and the author is an extreme element of the set of ``martingale'' couplings between two real probability measures in convex order. It enjoys remarkable properties with respect to order relations and a minimisation problem inspired by the theory of optimal transport. An explicit representation and a number of further noteworthy attributes have recently been established by Henry-Labord\`ere and Touzi. In the present paper we prove that the curtain coupling depends continuously on the prescribed marginals and quantify this with Lipschitz estimates. Moreover, we investigate the Markov composition of curtain couplings as a way of associating Markovian martingales with peacocks. 
\end{abstract}
\maketitle

\section*{Introduction}

There are at least two standard methods to couple real random variables, that is to obtain a joint law. The first one is the product (or independent coupling) $(\mu,\nu)\mapsto\mu\otimes\nu$, and the other is the quantile coupling $(\mu,\nu)\mapsto\law(G_\mu, G_\nu)$ where $G_\mu,G_\nu$ are the generalised inverse of the cumulative functions of $\mu,\nu$ (also called quantile functions, see Paragraph \ref{plusgrand}). One can easily convince one-self that both operators are continuous in the weak topology. In this paper we are interested in the continuity of another method, 
namely the left-curtain coupling $\pi_\lc=\curt(\mu,\nu)$. It was recently introduced in \cite{BJ} by Beiglb\"ock and the author\footnote{Right-curtain couplings can be defined symmetrically and the corresponding result can be deduced easily. In this paper \emph{curtain coupling} and \emph{monotone coupling} indicate {left-curtain couplings} and {left-monotone couplings} respectively} and further studied by Henry-Labord\`ere and Touzi \cite{HT}. As defined in \cite{BJ}, $\pi_\lc$ is the measure with marginals $\mu$ and $\nu$ such that for every $x\in\R$, the two marginals of $\pi_\lc|_{]-\infty,x]\times\R}$ are $\mu_{]-\infty,x]}$ and the so-called {\it shadow} (see Definition \ref{shad_thm}) of the latter measure in $\nu$. We advocated that under the additional constraint $\E(Y|X)=X$ on $\law(X,Y)$ (that can be satisfied  neither by $\law(X)\otimes\law(Y)$ nor by the quantile coupling, except in degenerated cases), $\pi_\lc$ can be considered as the more natural coupling of $\mu=\law(X)$ and $\nu=\law(Y)$. Indeed it is distinct from the quantile coupling but can be considered as its natural counterpart under the martingale constraint. Moreover it enjoys remarkable optimality properties with respect to the natural martingale variant of the usual transport problem on $\R$, the martingale transport problem that was introduced in the context of mathematical finance in \cite{HoNe12, BeHePe11, GaLaTo}. See Proposition \ref{synthese} for more details on $\pi_\lc$. One of our main results is that the operator $\curt:(\mu,\nu)\mapsto \pi_\lc$ is continuous. Furthermore, we quantify the continuity with Lipschitz estimates. 

Let us emphasise that in close situations the continuity of a coupling operator may also be false. For this purpose, we consider the Markov composition of two probability measures on $\R^2$. If $\pi$ and $\pi'$ have respectively marginals $\mu_1,\mu_2$ and $\mu_2,\mu_3$ the Markov composition $\pi\circ\pi'$ is the law of $(X_1,X_2,X_3)$ where $\law(X_i)=\mu_i$ for $i=1,2,3$ and $X_1,\,X_3$ are conditionally independent given $X_2$. As observed by Kellerer \cite{Ke72}, this composition $(\pi,\pi')\mapsto \pi\circ \pi'$ is not a continuous operator. This observation provides a bridge to the next topic of this paper, namely the recent elaborations on a famous theorem established by Kellerer also in \cite{Ke72}. In the latter paper the author restricts the joint laws $\pi,\pi'$ to a certain space $\mathcal{N}$ of couplings and proves the continuity of $(\pi,\pi')\mapsto \pi\circ \pi'$ on $\mathcal{N}\times \mathcal{N}$. This is the core of the proof in the celebrated Kellerer theorem that states that if $(\mu_t)_{t\geq 0}$ is non-decreasing in the convex order (see Definition \ref{orders} and the definitions of Chapter 3) there exists a {\it Markovian} martingale $(M_t)_{t\geq 0}$ with $\law(M_t)=\mu_t$ for every $t\geq 0$. In our article we roughly replace $\mathcal{N}$ with the space of left-curtain couplings and try to build a Markovian martingale with prescribed $1$-marginals such that, roughly speaking, the $2$-marginals between times $t$ and $t+\mathrm{d} t$ are left-curtain couplings. Because of the lack of continuity of the Markov composition for left-curtain couplings, we can not apply Kellerer's strategy to obtain the complete Kellerer theorem. Nevertheless we prove the Markovianity and also the uniqueness in specific cases. We also provide examples of non-Markovian martingales. Our approach that is detailed in the last chapter is parallel to the one by Henry-Labord\`ere, Tan, and Touzi \cite{HTT}. We explore new classes of examples and address new questions as the ones related to (non-)Markovianity. We postpone the discussion of the overlap as well as the differences to the respective chapter.

Let us write a short summary on the martingale transport problem seen from the perspective of the modern theory of optimal transportation on geodesic metric spaces. A first remark is that the uniqueness of the minimiser $\pi_\lc$ of the problem is typical in the theory. It is usual that the minimiser possesses some structure. In particular, the usual shape is $\pi=(\id\otimes T)_\#\mu$ where $\mu$ and $\nu=T_\#\mu$ are the marginals of the problem. Such structure results are usually called Brenier--McCann's theorems after \cite{Br2,McC}. The main result in \cite{BJ} shows a similar behaviour for continuous measures $\mu$ because the Markov kernel characterising $\pi_\lc$ is concentrated on two points. This is the minimal support for martingales. The dual theory of the problem can also be adapted as explained in \cite{BeHePe11}. Brenier--McCann theorems are the first step toward further developments. The next step is to consider probability measures $(\mu_t)_{t\in[0,1]}$ and in particular the {\it displacement interpolation} introduced by McCann in \cite{McCgas}. In this interpolation $\mu_0$ and $\mu_1$ are the given marginals of the transport problem. Instead of considering the usual convex interpolation $t\mapsto t\mu_1+(1-t)\mu_0$, McCann considers the displacement interpolation $t\mapsto \law(tX_1+(1-t)X_0)$ where the law of $(X_0,X_1)$ is the optimal transport plan. In the martingale setting, this interpolation does not provide a continuous martingale. The search for an appropriate martingale displacement interpolation may be related to the Skorokhod embedding problem. See \cite{BCH} for an approach to this classical problem using the optimal transport methods. Let us now relate the optimal transport theory to the peacock problem, which is the modern name given by Hirsch, Profeta, Roynette, and Yor in \cite{HPRY} for studies related to Kellerer's Theorem (see Chapter 3 and above). In fact an important result in the theory of optimal transportation is the possibility of deriving a process $(X_t)_t$ from any curve $(\mu_t)_{t\in[0,1]}$, $K$-Lipschitz with respect to the Wasserstein distance (the distance of optimal transport), such that for every $t\in[0,1]$, the mean quadratic speed $(\E(|\dot{X}_t|^2))^{1/2}$ is smaller than $K$ and $\law(X_t)$ is $\mu_t$. See \cite{Lis} for a more precise statement and the monograph \cite{AGS} where it plays the role of a key-result in the theory of gradient flows of functionals defined on Wasserstein spaces. Replacing a curve of probability measures with a probability measure on trajectories is also our topic in the present paper, but in the martingale setting.

The paper is organised in three quite separate chapters. The type of results and proofs are different. Nevertheless, lemmas introduced in part one are required in part two and the continuity of the left-curtain coupling, proved in part two is invoked for finitely supported measures in part three (see the discussion before Lemma \ref{lemlocal}). In the first part we introduce several notions related to positive measures on $\R$. In particular we introduce the convex order $\leqc$ and the extended order $\leqe$. We investigate these two orders as well as five other more or less classical orders on the space of finite real measures with finite first moment. We prove that they can all be formulated in terms of random variables and deduce relations between them. Theorem \ref{echelle} may be considered the main result of this chapter.

We start Chapter 2 with the definitions of the shadow projection (Definition \ref{shad_thm}) and the left-curtain coupling (Definition \ref{lc}), and recall their main properties. The main theorem and more satisfactory result of this part is Theorem \ref{lips}. It states that the shadow projection $(\mu,\nu)\mapsto S^{\nu}(\mu)$ is a Lipschitzian map for the Kantorovich metric $W$. Corollary \ref{doubidou} is a reformulation of this result for the left-curtain coupling $\curt$. Another important feature of this chapter is the development of a new modified support $\spt^*(\pi)$ in the theory of \cite{BJ}. It enhances the rough definition of left-monotone couplings that appeared in the equivalence between the three property for couplings: optimal, left-monotone or left-curtain (Proposition \ref{synthese} in the present paper). With Proposition \ref{car_etoile}, it is now possible to determine whether a coupling $\pi$ belongs to this category only by considering the triples of points in a well-defined set, namely $\spt^*(\pi)\subset\R^2$, while in the previous definition of left-monotone one had to show that {\it there exists} $\Gamma\subseteq \R^2$ of full $\pi$-measure with the desired property for all triples in $\Gamma^3$. A similar characterisation of left-monotone couplings is not available with $\Gamma=\spt(\pi)$ as explained in Example \ref{countex}. The developments of the object $\spt^*$ enables us in Theorem \ref{corstab} to extrapolate the usual proof of the continuity of the left-curtain coupling (see  \cite[Theorem 5.20]{Vi2}). However the quantitative Theorem \ref{lips} does not rely on it but on the important Lemma \ref{ahyo}, on the monotonicity of the shadow projection with respect to the stochastic order.

In Chapter 3 we consider the construction of a martingale $(M_t)_t$ fitting a given curve $(\mu_t)_t$ of probability measures, that is such that $\law(M_t)=\mu_t$ for every $t$. It is the famous peacock problem of Hirsch, Profeta, Roynette, and Yor \cite{HPRY} that we present extensively in the introduction of this chapter. Our method is similar to that of by Henry-Labord\`ere, Tan, and Touzi in \cite{HTT}. It consists of considering discretisations of $\mu_t$ and composing the transitions given by the left-curtain coupling using Markov composition. The resulting process is a piecewise constant Markovian martingale. Letting the mesh of the time partition tend to zero, one may obtain martingale processes that are or are not Markovian. In the final part we give some typical examples with several different behaviours. For instance in Theorem \ref{Poisson_fini} the measures $\mu_t$ are assumed to be finitely supported and we prove uniqueness and Markovianity. In Proposition \ref{reprise} we merely assume that $t\mapsto\mu_t$ is right continuous and, inspired by \cite{HiRo13}, prove that there exists at least one limit martingale for the finite dimensional convergence. We close the paper by suggesting open questions on this topic.

\section{Reminders about the stochastic and convex orders}\label{premier}

We consider the space $\m$ of positive measures on $\R$ with finite first moments. The subspace of probability measures with finite expectation is denoted by $\p$. For $\mu,\,\nu\in\m$, the Kantorovich distance defined by
\begin{align}\label{kanto}
W(\mu,\nu)=\sup_{f\in\mathrm{Lip}(1)}\left|\int\,f\,\dd\mu-\int\,f\,\dd\nu\right|
\end{align}
endows $(\p,W)$ with the usual topology $\mathcal{T}_1$ for probability measures with finite first moments. In the definition, the supremum is taken among all $1$-Lipschitzian functions $f:\R\to\R$. We also consider $W$ with the same definition on the the subspace $m\p=\{\mu\in\m|\,\mu(\R)=m\}\subset\m$ of measures of mass $m$.

According to the Kantorovich duality theorem, an alternative definition in the case $\mu,\nu\in\p$ is
\begin{align}\label{wass}
\inf_{(\Omega,X,Y)}\E(|Y-X|)
\end{align}
where $X,\,Y:(\Omega,\mathcal{F},\P)\to\R$ are random variables with marginals $\mu$ and $\nu$. The infimum is taken among all joint laws $(X,Y)$, the probability space $(\Omega,\mathcal{F},\P)$ being part of the minimisation problem. Note that without loss of generality $(\Omega,\mathcal{F},\P)$ can be assumed to be $([0,1],\mathcal{B},\lambda)$ where $\lambda$ is the Lebesgue measure and $\mathcal{B}$ the $\sigma$-algebra of Borel sets on $[0,1]$.

A special choice of 1-Lipschitzian function is the function $f_t:x\in\R\to |x-t|\in \R$. Therefore if $\mu_n\to\mu$ in $\m$, the sequence of functions $u_{\mu_n}:t\mapsto \int f_t(x)\dd\mu_n(x)$ pointwise converges to $u_\mu$. The converse statement also holds if all the measures have the same mass and barycenter (see \cite[Proposition 4.2]{BJ} or directly \cite[Proposition 2.3]{HiRo12}). For every $\mu\in\m$, the function $u_\mu$ is usually called the potential function of $\mu$. 

A measure $\pi$ on $\R^2$ is called a {\it transport plan} or a {\it coupling}. Let $\Pi(\mu,\nu)$ be the space of transport plans with marginals $\mu$ and $\nu$. The subspace $\M(\mu,\nu)$ is defined as follows
$$\M(\mu,\nu)=\{\pi=\law(X,Y)\in\Pi(\mu,\nu),\,\E(Y|X)=X\},$$
where the constraint $\E(Y|X)=X$ means: $\E(Y|X=x)=x$ for $\mu$-almost every $x\in \R$. 

We need to define $W^{\R^2}$, the Kantorovich metric on $\R^2$ in order to compare transport plans. It is defined identically to the $1$-dimensional version in \eqref{kanto} and \eqref{wass}, except that $|.|$ is replace with a norm $\|.\|$ of $\R^2$. Indeed the choice of a norm  is required in the definition of the $1$-Lipschitzian functions in \eqref{kanto} and more directly in \eqref{wass}. In the same way, we introduce $W^{\R^d}$ for the Euclidean spaces of greater dimension. It is a metric on $\p(\R^d)$, the space of probability measures with finite first moment. We denote by $\mathcal{T}_1(\R^d)$ the topology induced by $W^{\R^d}$ and $\mathcal{T}_{\mathrm{cb}}(\R^d)$ the usual weak topology.  The letters ``$\mathrm{cb}$'' stay for continuous bounded functions because they define the weak topology while the former topology is induced by the continuous functions growing at most linearly at infinity.

\begin{rem}\label{rem_topo}
Let us make precise what is the link between the topologies $\mathcal{T}_{\mathrm{cb}}(\R^d)$ and $\mathcal{T}_1(\R^d)$. As explained in \cite[Theorem 7.12 with the comments]{Vi1}, the two topologies coincide on each set $\mathcal{C}$ made of uniformly integrable measures, as for instance $\bigcup_k\{\gamma_k\}\cup\{\gamma\}$ where $(\gamma_k)_k$ weakly converges to $\gamma$. A consequence of this fact is that for a familly $\mathcal{C}_1,\ldots,\mathcal{C}_j$ of subsets of $\p$ of this type, the set $\mathcal{C}$ of measures $\pi$ with the $i$-th marginal in $\mathcal{C}_i$ for every $i\leq j$ satisfies itself the uniform integrability. In fact
$$\iiint_{|x_1|+\cdots+|x_j|\geq R}|x_1|+\cdots+|x_j|\,\dd\pi(x_1,\ldots,x_j)\leq \sum_{i=1}^j\int_{x_i\geq R/j}|x_i|\dd\mu_i(x_i)$$
where for every $i\leq j$, $\mu_i\in\mathcal{C}_i$ is the $i$-th marginal of $\pi$. This tends to zero uniformly on $\mathcal{C}\ni \pi$ when $R$ goes to infinity.
\end{rem}

\subsection{Seven partial orders on $\m$.}
We introduce seven partial orders on $\m$, investigate their dependance, and explain their meaning in terms of couplings. These definitions will be useful for a synthetic formulation in Chapter 2, like for instance in Lemma \ref{add_queue}. The results of this chapter continue the extension of the convex order started with the extended order in \cite{BJ} to other cones of functions. They are applied in Chapter 2 but may also be interesting in themselves. Even if the results like Theorem \ref{echelle} and Corollary \ref{coro} may sound classical and the proofs are easy, they are to our knowledge the first apparition in the literature.
\begin{defi}\label{orders}
The letter $E$ is a variable for a set of real functions growing linearly at most in $-\infty$ and $+\infty$. We introduce the set of non-negative functions $E_+$, the set of non-increasing functions $E_\sto$ and the set of convex functions $E_C$, all three are restricted to functions with the growing constraint. For $\mu,\nu\in\m$ we introduce the property $P(E)$.
$$P(E):\quad\forall \phi\in E,\, \int \phi\,\dd\mu\leq\int \phi\,\dd\nu.$$
For $\mu,\,\nu\in\m$,
\begin{itemize}
\item if $P(E_+)$ holds, we write $\mu\leqp \nu$ (usual order),
\item if $P(E_\sto)$ holds, we write $\mu\leqs \nu$ (stochastic order or first order stochastic dominance),
\item if $P(E_C)$ holds, we write $\mu\leqc \nu$ (convex order, Choquet order or second order stochastic dominance),
\item if $P(E_{C}\cap E_+)$ holds, we write $\mu\leqcp \nu$,
\item if $P(E_+\cap E_\sto)$ holds, we write $\mu\leqps \nu$,
\item if $P(E_C\cap E_\sto)$ holds, we write $\mu\leqcs \nu$,
\item if $P(E_C \cap E_+ \cap E_\sto)$ holds, we write $\mu\leqcps \nu$. 
\end{itemize}
\end{defi}

\begin{rem}[Usual notations]\label{extended}
The usual notation for $\mu\leqp\nu$ is $\mu\leq \nu$. In \cite{Th00}, $\stackrel{D}{\leq}$ is the notation for the stochastic order $\leqs$. In \cite{Ke72}, the author simply denotes $\leqcs$ by $\prec$. In \cite{BJ}, Beiglb\"ock and the author introduced the extended order $\preceq_E$. The latter is the same as $\leqcp$ in this paper.
\end{rem}

\subsection{Complements to the stochastic order}\label{plusgrand}
Recall that the Lebesgue measure is denoted by $\lambda$. For a measure $\nu$, we note $F_\nu$, the cumulative distribution function and $G_\nu$, the quantile function. Recall that $G_\nu(t)=\inf_{x\in \R}\{F_\nu(x)\geq t\}$. This function can be seen as a general inverse of $F_\nu$. It is  left-continuous and defined on $[0,\nu(\R)]$. Recall also $\nu=(G_\nu)_\#\lambda|_{[0,\nu(\R)]}$, which will be used extensively in this paper. 

The following standard proposition can for instance be found in \cite[Theorem 3.1]{Th00}. See also the introduction of paragraph \ref{comp_conv}. The proof makes use of the quantile functions.
\begin{pro}\label{ordre_sto}
For $\mu,\,\nu\in\p$, the relation $\mu\leqs\nu$ holds, if and only if there exists a pair of random variables $(X,Y)$ on a probability space $(\Omega,\mathcal{F},\P)$ with marginals $\mu$ and $\nu$, such that $X\leq Y$, $\P$-almost surely.
\end{pro}
We can actually choose $\P=\lambda_{[0,1]}$, $X=G_\mu$ and $Y=G_\nu$. Furthermore, note that with this representation the pair $(X,Y)$ gives the minimal value in \eqref{wass}. Indeed the bound $|\E(Y)-\E(X)|\leq\E(|Y-X|)$ is always satisfied but if $X\leq Y$ we also have $\E(|Y-X|)=\E(Y)-\E(X)$. Actually we have more generally
\begin{lem}\label{stokanto}
Let $\mu,\,\nu$ be in $\p$. The coupling $(G_\mu,G_\nu)$ defined on $\Omega=([0,1],\mathcal{B},\Lg)$ is optimal in the definition \eqref{wass} of $W(\mu,\nu)$. More generally if $\mu,\,\nu$ have mass $m\neq 1$ we have also
$$W(\mu,\nu)=\int |G_\nu-G_\mu|\dd\lambda_{[0,m]}=\|G_\nu-G_\mu\|_1.$$
Moreover if $\mu\leqs \nu,$ 
$$W(\mu,\nu)=\int_0^m (G_\nu-G_\mu)\dd\lambda= m\left(\frac1m\int x\,\dd\mu-\frac1m\int x\,\dd\nu\right).$$
\end{lem}

Let us define the {\it rightmost} and {\it leftmost measure of mass $\alpha$ smaller that $\nu$}. Denoting the mass of $\nu$ by $m$ and assuming $\alpha\leq m$, we consider the set $S=\{\mu\in\m|\,\mu(\R)=\alpha\text{ and }\mu\leqp\nu\}$. Let us prove that for any $\mu\in S$, we have $\mu\leqs \nu_{\alpha}$ where $\nu_\alpha$ denotes $G_\#\lambda|_{[m-\alpha,m]}$.
Let $\varphi:\R\to\R$ be a non-decreasing function, integrable for the elements of $\m$. Hence
\begin{align*}
\int \varphi\,\dd \mu&=\int \varphi\, \dd(\mu\wedge\nu_\alpha)+\int \varphi\, \dd[\mu-(\mu\wedge\nu_\alpha)]\\
&\leq \int \varphi \,\dd(\mu\wedge\nu_\alpha)+\varphi(G_\nu(m-\alpha))[\mu-(\mu\wedge\nu_\alpha)](\R)\\
&\leq \int \varphi \,\dd(\mu\wedge\nu_\alpha)+\varphi(G_\nu(m-\alpha))[\nu_\alpha-(\mu\wedge\nu_\alpha)](\R)\\
&\leq \int \varphi\,\dd \nu_\alpha.
\end{align*}
Indeed $\nu_\alpha-\mu$ admits a density with respect to $\nu$ that is non-positive on $]-\infty,G_\nu(m-\alpha)]$ and non-negative on $[G_\nu(m-\alpha),+\infty[$. The measure $\nu_\alpha$ is \emph{the rightmost measure of mass $\alpha$ smaller than $\nu$}. Symmetrically $(G_\nu)_\#\lambda|_{[0,\alpha]}$ is the leftmost measure.

\subsection{Complements to the convex order}\label{comp_conv}
 
In \cite[Theorem 8]{St65}, Strassen establishes a statement on the marginals of $k$-dimensional martingales indexed on $\N$. For our purposes, we restrict the statement to $1$-dimensional martingales with one time-step.  This result is related to the convex order $\leqc$ in the same way as Proposition \ref{ordre_sto} is associated with $\leqs$. Actually, in particular for more general ordered spaces than $\R$, Proposition \ref{ordre_sto} is widely referred to as Strassen's Theorem on stochastic dominance. The theorem is attributed to Strassen because of \cite{St65}. However, the statement of this result in the paper by Strassen is very elusive. It corresponds to two lines on page 438 after the proof of Theorem 11. See a paper by Lindvall \cite{Lin}, where a proof relying on Theorem 7 by Strassen is restituted with all the details. Therefore, we prefer to reserve the name Strassen's Theorem for the domination in convex order and we later call similar results, like Proposition \ref{ordre_sto}, Strassen-type theorems.
\begin{pro}[Theorem of Strassen]\label{strassen1}
For $\mu,\,\nu\in\p$, the relation $\mu\leqc\nu$ holds if and only if there exists a pair of random variables $(X,Y)$ on a probability space $(\Omega,\mathcal{F},\P)$ with marginals $\mu$ and $\nu$, such that $\E(Y|\,X)=X$, $\P$-almost surely.
\end{pro}

In the same article \cite[Theorem 9]{St65} Strassen states a result on submartingales that we recall for submartingales indexed on two times.
\begin{pro}[Theorem of Strassen 2]\label{strassen2}
For $\mu,\,\nu\in\p$, the relation $\mu\leqcs\nu$ holds if and only if there exists a pair of random variables $(X,Z)$ on a probability space $(\Omega,\mathcal{F},\P)$ with marginals $\mu$ and $\nu$, such that $\E(Z|\,X)\geq X$, $\P$-almost surely.
\end{pro}

Note that if we introduce $Y=\E(Z|\,X)$, one has $\mu\leqs\law(Y)$ and $\law(Y)\leqc\nu$. This kind of decomposition will be investigated in the next section.

\subsection{Strassen-type theorems}
Before we state Theorem \ref{echelle}, let us clarify a point of notation. One may permute the subscripts of $\preceq$ without changing the meaning of the partial orders. For instance $\preceq_{+,\sto,C}$ does not appear in Definition \ref{orders} but it denotes the same order as $\leqcps$. More than one notation for the same object seems useless but the arrangement of the indices makes sense in the following theorem.

\begin{them}[Chain of relations]\label{echelle}All the relations of Definition \ref{orders} are antisymmetric and transitive, making them partial orders.

Moreover, for any sequence $(\mu_i)_{i=0,\ldots,n}$ (with $n=2$ or $3$) satisfying the relations $\mu_{i-1}\preceq_{r_i}\mu_{i}$ for $i=1,\ldots,n$ one has $\mu_0\preceq_{r_1,\ldots,r_n}\mu_n$.
 
Conversely if $\mu_0\preceq_{r_1,\ldots,r_n}\mu_n$ one can find a sequence $(\mu_i)_{i=0,\ldots,n}$ such that $\mu_{i-1}\preceq_{r_i}\mu_{i}$ for every $i\geq 1$.
\end{them}
\begin{proof}
1. The transitivity is obvious. For the reflexivity, it is enough to prove that $\leqcps$ is reflexive. Let $\mu$ and $\nu$ satisfy $\mu\leqcps \nu$ and $\nu\leqcps \mu$. Hence integrating with respect to $\mu$ or $\nu$ provides the same value for any function that can be written in all three forms --- (i) the difference of two non-negative functions, (ii) the difference of two non-decreasing functions, (iii) the difference of two convex functions. All the three spaces are restricted to functions growing at most linearly in $\pm \infty$. Continuous piecewise affine functions with finitely many pieces satisfy the three conditions. Thus $\mu=\nu$.

2. The first implication is obvious, the converse statement is not. We have to prove it for twelve different partial orders. For $\leqcp$ (see Remark \ref{extended}) we simply quote \cite[Proposition 4.4]{BJ}. From this, we can easily deduce the statement for $\preceq_{+,C}$. We consider $\mu_0\preceq_{+,C}\mu_1$. As the order is the same as $\leqcp$, we can find $\mu_1$ with $\mu_0\leqc \mu_1$ and $\mu_1\leqp \mu_2$. We set $\mu'_1=\mu_0+(\mu_2-\mu_1)$. As $\mu_2-\mu_1$ is a positive measure one has $\mu_0\leqp\mu'_1$. Let $\varphi$ be a convex function. Therefore
\begin{align*}
\int \varphi \dd\mu'_1&=\int\varphi \dd\mu_0+\int\varphi \dd(\mu_2-\mu_1)\\
&\leq\int\varphi \dd\mu_1+\int\varphi \dd(\mu_2-\mu_1)\leq \int \varphi \dd\mu_2,
\end{align*}
which means $\mu_1'\leqc \mu_2$.
The last argument can be used for stating the decomposition of $\preceq_{+,C,\sto}$ and $\preceq_{C,+,\sto}$ provided we can prove it for $\preceq_{C,\sto,+}$. The place of the index ``$+$'' does not matter. Similarly the decomposition of $\preceq_{+,\sto,C}$ and $\preceq_{\sto,+,C}$ will be a corollary of the property for $\preceq_{\sto,C,+}$.  In the same way $\preceq_{+,\sto}$ reduces to the study of $\leqps$. 

3. We prove here the two wanted decompositions of $\mu\leqcs\nu$. For probability measures, the Strassen theorem (Proposition \ref{strassen2}) states that there exists $(X,Z)$ with $\law(X)=\mu$, $\law(Z)=\nu$ and $Y:=\E(Z|\,X)\geq X$. For $\mu_1$ defined as the law of $Y$ and $\mu'_1$ as the law of $Y':=Z-(Y-X)$ we have $\mu=\mu_0\leqs\mu_1\leqc\mu_2=\nu$ and $\mu=\mu_0\leqc\mu'_1\leqs\mu_2=\nu$. If $\mu,\,\nu$ are not probability measures, they must have the same mass. Indeed, every constant function is element of $E_C\cap E_\sto$. One can easily obtain the statement by normalising the measures.

4. We are left with $\preceq_{\sto,+}$, $\preceq_{\sto,C,+}$ and $\preceq_{C,\sto,+}$. Having in mind the possibility to transpose ``$C$'' and ``$\sto$'' proved in the last paragraph, it is sufficient to consider $\mu\preceq_{\sto,+}\nu$ and $\mu\preceq_{\sto,C,+}\nu$. For that purpose we consider $\nu'=(G_\nu)_\#\lambda|_{[\nu(\R)-\mu(\R),\nu(\R)]}$. Recall that it is the rightmost measure of mass $\mu(\R)$ smaller than $\nu$ introduced in paragraph \ref{plusgrand}. Of course $\nu'\leqp\nu$. We now prove $\mu\leqs \nu'$ and $\mu\preceq_{\sto,C} \nu'$ respectively. Let $\varphi\in E$ with $E=E_\sto$ or $E=E_\sto\cap E_C$ respectively. Because of the dominated convergence theorem, we can assume without loss of generality that $\varphi$ is bounded from below. We denote $G_\nu(\nu(\R)-\mu(\R))$ by $x\in[-\infty,+\infty[$ so that $\varphi-\varphi(x)$ is non-negative on $]x,+\infty[$. For simplicity, one considers that $\mu$ is a probability measure. By applying $\mu\leq_{\sto,+} \nu$ or $\mu\leq_{\sto,C,+} \nu$ for $(\varphi-\varphi(x))\chi_{[x,+\infty[}$ respectively, one obtains
\begin{align*}
\int \varphi\,\dd \mu&=\varphi(x)+\int[\varphi-\varphi(x)]\chi_{[x,+\infty[}\,\dd \mu\\
&\leq \varphi(x)+\int [\varphi-\varphi(x)]\chi_{[x,+\infty[}\,\dd \nu=\int\varphi\,\dd \nu'.
\end{align*}
Hence $\mu\leqs\nu'$ and $\mu\preceq_{\sto,C} \nu'$ respectively. For the latter we recall point 3 so that we have $\mu=\mu_0\leqs\mu_1\leqc\mu_2=\nu'$ and $\mu_2\leqp\mu_3=\nu$ for some intermediate measure $\mu_1$.
\end{proof}

Theorem \ref{echelle} opens the door to a translation of all the partial orders in Definition \ref{orders} in terms of couplings. For this purpose we use what is known for $\leqs$ and $\leqc$ (Proposition \ref{ordre_sto} and Proposition \ref{strassen1}) together with the following characterisation: if $\nu\in\p$ then $\mu\leqp\nu$ if and only if there exists a random variable $Y$ defined on a probability space $(\Omega,\mathcal{F},\P)$ and an event $A$ such that $\mu(\R)=\P(A)$ and $\law(Y|\,A)=\mu(\R)^{-1}\mu$. The statement also requires the composition of joint laws, called gluing lemma in \cite{Vi2}. As an example let us reprove the converse statement of Proposition \ref{strassen2}. We start with $\mu_0,\,\mu_2\in\p$ satisfying $\mu_0\leqcp\mu_2$. With Theorem \ref{echelle}, we find $\mu_1$ satisfying $\mu_0\leqc\mu_1$ and $\mu_1\leqp\mu_2$. Hence on some probability space $\Omega_X$ we have a coupling $(X_0,X_1)$ of $\mu_0$ and $\mu_1$ that satisfies $\E(X_1|\,X_0)=X_0$ and on some probability space $\Omega_Y$ we have a coupling $(Y_1,Y_2)$ of $\mu_1$ and $\mu_2$ that satisfies $Y_1\leq Y_2$. Therefore by using the Markov composition, or the gluing lemma \cite[Chapter 1]{Vi2}, there exists some probability space $\Omega_Z$ and $(Z_0,Z_1,Z_2)$ such that $\law(Z_0,Z_1)=\law(X_0,X_1)$ and $\law(Z_1,Z_2)=\law(Y_1,Y_2)$. It follows $\E(Z_2|\,Z_0)\geq \E(Z_1|\,Z_0)=Z_0$. 

We give another illustration on how to apply Theorem \ref{echelle} in the case of an order made of three subscripts.

\begin{cor}\label{coro}
Let $\mu,\nu$ be elements of $\m$. The relation $\mu\leqcps \nu$ holds if and only if there exists a probability space $(\Omega,\mathcal{F},\P)$ with a measurable set $A$ and two random variables $(X,T)$ satisfying
$$\P_A\text{-almost surely }X\leq \E(T|\,X,A)$$
where
\begin{enumerate}
\item $\law(X|\,A)=\mu(\R)^{-1}\mu$
\item $\law(T)=\nu(\R)^{-1}\nu$,
\item $\P(A)=\mu(\R)\nu(\R)^{-1}$,
\end{enumerate}
\end{cor}
\begin{proof}
1. According to Theorem \ref{echelle}, setting $\mu_0=\mu$ and $\mu_3=\nu$, we can find $\mu_1,\,\mu_2\in\m$ with $\mu_0\leqc \mu_1\leqp\mu_2\leqs\mu_3$. We first assume $\mu_3\in\p$  for simplicity. We apply Proposition \ref{ordre_sto} and Proposition \ref{strassen1} to the pairs $(\mu_2,\mu_3)$ and $(\mu_0,\mu_1)$. According to the usual compositions rules of the probability theory, we can find a pair $(Z,T)$ for $(\mu_2,\mu_3)$ and $(X,Y)$ for $(\mu_0(\R)^{-1}\mu_0,\mu_1(\R)^{-1}\mu_1)$ satisfying the relations explained in these propositions. Moreover usual properties of the probability theory allow us to couple these random variables in a probability space $(\Omega,\mathcal{F},\P)$ and its restriction $(A,\mathcal{F}_A,\P_A)$ where $A\subseteq \Omega$ is a Borel set adapted to the relation $\mu_1\leqp\mu_2$. It satisfies $\P(A)\mu_2(\R)=\mu_1(\R)$ and we have
\begin{itemize}
\item $\law(T)=\mu_3$
\item $\law(Z)=\mu_2$
\item $\law(Y)=\P(A)^{-1}\mu_1$
\item $\law(X)=\P(A)^{-1}\mu_0$
\end{itemize}
and
\begin{itemize}
\item $Z\leq T$
\item $Y=Z$, $\P_A$-almost surely
\item $X=\E_A(Y|\,X)$, $\P_A$-almost surely
\end{itemize}
The last line also writes $X=\E(Z|\,X,A)$, $\P_A$-almost surely. Thus $X\leq\E(T|\,X,A)$, $\P_A$-almost surely.

2. We prove the converse statement. We assume that $\P_A\text{-almost surely }X\leq \E(T|\,X,A)$ is satisfied and consider $\varphi\in E_C\cap E_+ \cap E_\sto$. We have
$\E_A(\varphi(X))\leq \E_A(\varphi(\E(T|\,X,A)))$ because $\varphi$ is non-decreasing. This is smaller than $\E_A(\varphi(T))$ because $\varphi$ is convex. Finally this is smaller than $\P(A)\E(\varphi(T))$ because $\varphi$ is non negative. We conclude with (1)---(3) that $\int\varphi\,\dd \mu\leq\int\varphi\,\dd \nu$.

3. The statement is established if $\nu=\mu_3$ is a probability measure. Using the usual normalisation of finite measures to probability measures, we get the other cases.
\end{proof}

\begin{rem}
The Strassen-type theorems admit equivalent translations in terms of transport of measure in place of couplings of random variables. Indeed, starting with a relation $\mu\preceq \nu$ and after using the decomposition provided by Theorem \ref{echelle}, the translation of each single relation can be made as follow : the relation $\leqs$ means that the elements of mass are transported in the direction of $+\infty$. The relation $\leqp$ means that some mass is created. Finally $\leqc$ denotes a dilation: each element of mass in position $x$ is spread in both directions in a way such that for any $x$ the barycenter of the mass transported from $x$ is still $x$.
\end{rem}

\section{Lipschitz continuity of the curtain coupling with respect to its marginals}\label{lips_esti}

In this section we recall the properties of the martingale curtain coupling $\pi_\lc=\curt(\mu,\nu)$ between two measures $\mu\leqc \nu$. We prove that it is a continuous map by using the property of monotonicity satisfied by curtain couplings. We establish a Lipschitz estimate for the shadow projection $(\mu,\nu)\mapsto S^\nu(\mu)$ and deduce that $\curt:(\mu,\nu)\in\p\times\p\longrightarrow \M$ is Lipschitzian when $\M$ is considered with the ad hoc (semi)metric $Z$. We also prove that such an estimate does not hold in $(\M,W^{\R^2})$. An important mathematical object introduced in this chapter is the reduced support that we denote $\spt^*\pi$. This set of full mass contributes to a better understanding of the property of monotonicity.

\subsection{Definitions of the shadows and the curtain coupling}

In \cite[Lemma 4.6]{BJ} the following important theorem-definition is proven.
\begin{defi}[Definition of the shadow]\label{shad_thm}
If $\mu\leqcp \nu$, there exists a unique measure $\eta$ such that
\begin{itemize}
\item $\mu\leqc \eta$
\item $\eta\leqp \nu$
\item If $\eta'$ satisfies the two first conditions (i.e $\mu\leqc \eta'\leqp \nu$), one has $\eta\leqc \eta'$.
\end{itemize}
This measure $\eta$ is called the shadow of $\mu$ in $\nu$ and we denote it by $S^\nu(\mu)$.
\end{defi}
The shadows are sometimes difficult to determine. An important fact is that they have the smallest variance among the set of measures $\eta'$. Indeed, $\eta\leqc\eta'$ implies $\int x \dd\eta=\int x \dd\eta'$ and $\int x^2 \dd\eta\leq\int x^2 \dd\eta'$ with equality if and only if $\eta=\eta'$ or $\int x^2 \dd\eta=+\infty$. 
\begin{ex}[Shadow of an atom, Example 4.7 in \cite{BJ}]\label{un_atome}
Let $\delta$ be an atom of mass $\alpha$ at a point $x$. Assume that $\delta\leqe \nu$. Then $S^\nu(\delta)$ is the restriction of $\nu$ between two quantiles, more precisely it is
$\nu'=(G_\nu)_\#\lambda_{]s;s'[}$ where $s'- s =\alpha$ and the barycenter of $\nu'$ is $x$.
\end{ex}
The next lemma describes the tail of the shadows. 
\begin{lem}\label{ombre_a_droite}
Let $\mu,\,\nu\in \m$ satisfy $\mu\leqcp \nu$. Assume that $y=\sup[\spt \mu]$ is finite. Then the restriction of $(S^\nu(\mu)-\mu)_+$ to $[y,+\infty[$ is the stochastically leftmost measure $\theta$ among the measures of the same mass satisfying $\theta\leqp(\nu-\mu)_+|_{[y,+\infty[}$.

The corresponding statement holds in the case $\inf[\spt \mu]>-\infty$.
\end{lem}
Before we write the proof, let us make clear that if $\nu$ has no atom in $y$, the measures $(S^\nu(\mu)-\mu)_+|_{[y,+\infty[}$
is simply $S^\nu(\mu)|_{[y,+\infty[}$ while $(\nu-\mu)_+|_{[y,+\infty[}=\nu_{[y,+\infty[}$.
\begin{proof}
Using Strassen's Theorem (Proposition \ref{strassen1}), let $\pi$ be a martingale transport plan with marginals $\mu$ and $S^\nu(\mu)$. Let $(\pi_x)_{x\in\R}$ be a disintegration where the measures $\pi_x$ are probability measures. Each $\pi_x$ can again be disintegrated in a family of probability measures concentrated on two points and with barycenter $x$. Observe now that for $a<x<b$ and $b'\in]x,b]$, one can compare $\frac{b-x}{b-a}\delta_a+\frac{x-a}{b-a}\delta_b$ with $\frac{b'-x}{b'-a}\delta_a+\frac{x-a}{b'-a}\delta_{b'}$ in the following way:
\begin{itemize}
\item both measures have mass $1$ and barycenter $x$,
\item $\frac{b'-x}{b'-a}\delta_a\leqp\frac{b-x}{b-a}\delta_a$ (inequality for the mass in $a$),
\item $\frac{b'-x}{b'-a}\delta_a+\frac{x-a}{b'-a}\delta_{b'}\leqc\frac{b-x}{b-a}\delta_a+\frac{x-a}{b-a}\delta_b$. 
\end{itemize}
Remind that $S^\nu(\mu)=\int [\frac{b-x}{b-a}\delta_a+\frac{x-a}{b-a}\delta_b]\, \dd \zeta_0(x,a,b)$ where $\zeta_0$ is a positive measure with first marginal $\mu$ that is concentrated on $\{(x,a,b)\in\R^3,\,a< x< b\text{ or }a=x=b\}$. For $a=x=b$, we adopt the convention $\frac{b-x}{b-a}\delta_a+\frac{x-a}{b-a}\delta_b=\delta_x$. The measure $(S^\nu(\mu)-\mu)_+|_{[y,+\infty[}$ of the statement can be written $\theta=\int [\frac{x-a}{b-a}\delta_b]\, \dd \zeta(x,a,b)$ where $\zeta\leqp \zeta_0$ and $\zeta$ is concentrated on $\{(x,a,b)\in\R^3,\,x<b\}$. Let $\theta'$ satisfy $\theta'\leqp(\nu-\mu)_+|_{[y,+\infty[}$ and $\theta'\leqs\theta$. Hence one can consider a measure $\bar\zeta$ concentrated on $\{(x,a,b,b')\in\R^4,\,x\leq b'<b\}$ such that $\theta'=\int [\frac{x-a}{b-a}\delta_{b'}]\, \dd \bar\zeta(x,a,b,b')$ and the projection of $\bar\zeta$ on the three first coordinates is $\zeta$. We denote by $\zeta'$ the measure $\frac{b'-a}{b-a}\bar\zeta$ and with a slight abuse of notation we denote also by $\zeta'$ its projection on the first three coordinates. We set 
\begin{align*}
\eta=&\int \left[\frac{b'-x}{b'-a}\delta_a+\frac{x-a}{b'-a}\delta_{b'}\right]\, \dd \zeta'(x,a,b,b')\\
&+\int \left[\frac{b-x}{b-a}\delta_a+\frac{x-a}{b-a}\delta_b\right]\,\dd (\zeta_0-\zeta')(x,a,b)
\end{align*}
Recall that
\begin{align*}
S^\nu(\mu)=&\int \left[\frac{b-x}{b-a}\delta_a+\frac{x-a}{b-a}\delta_{b}\right] \,\dd \zeta'(x,a,b,b')\\
&+\int \left[\frac{b-x}{b-a}\delta_a+\frac{x-a}{b-a}\delta_b\right] \,\dd (\zeta_0-\zeta')(x,a,b).
\end{align*}
Therefore according to the three remarks above, one has
\begin{itemize}
\item $\mu\leqc \eta$,
\item $\eta\leqp \nu$,
\item $\eta\leqc S^\nu(\mu).$
\end{itemize}
The second relation relies on 
$$\int \left[\frac{b'-x}{b'-a}\delta_a+\frac{x-a}{b'-a}\delta_{b'}\right] \dd \zeta'(x,a,b,b')= \int \left[\frac{b'-x}{b-a}\delta_a\right]d\zeta'+\theta'.$$
The last relation is in fact an equality. Indeed, the domination $\eta\geqc S^\nu(\mu)$ is a consequence of the two first relations and the definition of the shadow. Moreover $\leqc$ is antisymmetric so that $\eta=S^\nu(\mu)$. Hence $\zeta'$-almost surely we have $b=b'$, which means $\theta'=\theta$. We have proven that the restriction of $(S^\nu(\mu)-\mu)_+$ to $[y,+\infty[$ is the stochastically leftmost  measure smaller than $(\nu-\mu)_+|_{[y,+\infty[}$.
\end{proof}

The following result is one of the most important on the structure of shadows. It is Theorem 4.8 of \cite{BJ}.
\begin{pro}[Structure of shadows]\label{shadowsum}
Let $\gamma_1,\gamma_2$ and $\nu$ be elements of $\m$ and assume that $\mu=\gamma_1+\gamma_2\leqcp\nu$. Then we have $\gamma_2\leqcp \nu-S^\nu(\gamma_1)$ and
$$S^\nu(\gamma_1+\gamma_2)=S^\nu(\gamma_1)+S^{\nu-S^\nu(\gamma_1)}(\gamma_2).$$
\end{pro}

\begin{ex}[Shadow of a finite sum of atoms]\label{more_atoms}
Let $\mu$ be the measure $\sum_{i=1}^n \alpha_i\delta_{x_i}$ and $\nu=G_\#\lambda_{]0,m]}$ such hat $\mu\leqcp\nu$. We can apply Proposition \ref{shadowsum} to this sum as well as Example \ref{un_atome} on the shadow of one atom. We obtain recursively the following description. There exists an increasing sequence of sets $J_1\subseteq\cdots J_n\subseteq ]0,m]$ satisfying that $J_k$ has measure $\sum_{i=1}^k \alpha_i$ and $J_k\setminus J_{k-1}$ is a pseudo-interval of $]0,m]\setminus J_{k-1}$, that is $J_k\setminus J_{k-1}=]s,t]\setminus J_{k-1}$ for some $0\leq s,t\leq m$. These pseudo-intervals satisfy $S^\nu(\sum_{i=1}^k\alpha_i\delta_{x_i})=G_\#\lambda_{J_k}$ for every $k\leq n$.

Conversely any increasing sequence $(J_i)_{i=1,\ldots, n}$ such that $J_{k}\setminus J_{k-1}$ is a pseudo-interval of $]0,m]\setminus J_{k-1}$ is associated with a family of atoms $\alpha_i\delta_{x_i}$ with $\alpha_i=\lambda(J_i)-\lambda(J_{i-1})$ and $x_i$ is the barycenter of $G_\#\lambda_{J_i\setminus J_{i-1}}$ such that $G_\#\lambda_{J_k}$ is the shadow of $\sum_{i=1}^k \alpha_i\delta_{x_i}$ in $\nu$.
\end{ex}

With the shadow projections, we can introduce the left-curtain coupling. For atomic measures it is related to Example \ref{more_atoms} when we assume that $(x_i)_i$ is an increasing sequence.
\begin{defi}[Left-curtain coupling, Theorem 4.18 in \cite{BJ}]\label{lc}
Let $\mu,\,\nu\in \m$ satisfy $\mu\leqc \nu$. There exists a unique measure $\pi\in \M(\mu,\nu)$ such that for any $x\in \R$ the measure $\pi_{]-\infty,x]\times \R}$ has first marginal $\mu_{]-\infty,x]}$ and second marginal $S^\nu(\mu_{]-\infty,x]})$. We denote it by $\pi_{\lc}$ and call it \emph{left-curtain coupling}.
\end{defi}

One of the main theorems of \cite{BJ} is the equivalence of three properties of couplings: left-curtain, left-monotone and optimal. Let us define left-monotone couplings.
\begin{defi}[Left-monotone coupling]\label{lm}
Let $\pi$ be an element of $\M(\mu,\nu)$. The coupling $\pi$ is {\it left-monotone} if there exists a Borel set $\Gamma$ with
\begin{itemize}
\item $\pi(\Gamma)=1$
\item for every $(x,y^-)$, $(x,y^+)$ and $(x',y')$ elements of $\Gamma$ satisfying $x<x'$ and $y^-<y^+$, the real $y'$ is not an element of $]y^-,y^+[$. 
\end{itemize}
\end{defi}

We can now state the result.

\begin{pro}[Theorem 1.9 in \cite{BJ}]\label{synthese}
Let $\pi\in \M(\mu,\nu)$. We introduce $c:(x,y)\in\R^2\to [1+\tanh(-x)]\sqrt{y^2+1}$. The properties are equivalent.
\begin{itemize}
\item Left-curtain: the transport plan $\pi$ is the left-curtain coupling,
\item Left-monotone: the transport plan $\pi$ is left-monotone,
\item Optimal: for any $\tilde\pi\in\M(\mu,\nu)$, if $\tilde\pi\neq \pi$, then $\int c\,\dd\pi<\int c\,\dd\tilde\pi$
\end{itemize}
\end{pro}

\begin{rem}
See Example \ref{countex} about  the fact that the left-monotonicity may not be satisfied for $\Gamma=\spt\pi$ even if it is realised for another $\Gamma$. 
\end{rem}

\begin{rem}
Actually Theorem 1.9 in \cite{BJ} is written for another kind of cost $c$. However replacing Theorem 6.1 by Theorem 6.3, both of this paper, leads to this version. Actually the latter theorem states that if $c$ is defined as $(x,y)\mapsto\varphi(x)\psi(y)$ where $\varphi$ is positive and decreasing, $\psi$ is positive and strictly convex the implication ``optimal $\Rightarrow$ left-curtain'' still holds provided $\min_{\tilde\pi\in\M(\mu,\nu)}\int c\,\dd\tilde\pi$ is finite. In Proposition \ref{synthese} this condition is satisfied without more assumptions because $\mu,\,\nu$ have finite first moments and the given $c$ grows at most linearly in $\pm\infty$.

In \cite{HT}, Henry-Labord\`ere and Touzi have proved that functions $c$ such that the partial derivative $\partial_{yyx}c$ is identically negative also lead to the left-curtain coupling if $\min_{\tilde\pi\in\M(\mu,\nu)}\int c\,\dd\tilde\pi$ is finite. This contains, in the case of smooth functions $c$, both the kind of costs in \cite[Theorem 6.1]{BJ} and \cite[Theorem 6.3]{BJ}.
\end{rem}

\subsection{Qualitative continuity of the curtain coupling map}\label{quali}
In this paragraph we show that $\mathrm{Curt}:(\mu,\nu)\mapsto\pi_\lc$ is continuous. We are using the second property of left-curtain couplings: according to Proposition \ref{synthese} they are the left-monotone couplings. The next example illustrates that for a left-monotone $\pi$ the set $\Gamma=\spt(\pi)$ may not fulfil the desired properties in Definition \ref{lm}, which contrasts with the support in the classical transport problem. The modified support $\spt^*(\pi)$ that we define below does not suffer from this difficulty.
\begin{ex}\label{countex}
Consider $\mu=(1/2)\Lg_{[-1,1]}$ and $\nu=(\delta_{-1}+2\delta_0+\delta_1)/4$. For these marginals, considering the transport plan given by the left-curtain coupling, the mass contained in $[-1,0]$ is mapped to $\{-1,0\}$ while the mass in $[0,1]$ is mapped to $\{0,1\}$. Thus $(0,-1)$, $(0,1)$ and $(1,0)$ are elements of $\spt(\pi)$.
\end{ex}
This example is typical for difficulties that may arise on the diagonal set $\{(x,y)\in\spt(\pi),\,y=x\}$, for instance for points $(x,x)$ satisfying $u_\mu(x)=u_\nu(x)$. Here $(0,0)$ is such a point. In the Proposition \ref {car_etoile} we will see how to fix this problem with the reduced support $\spt^*(\pi)$ that we define now.

First, let $A$ be the set of $x\in\R$ such that $\pi(]-\infty,x[\times]x,\infty[)=0$. Second, we denote the subset of $A$ of points that are isolated in $A$ on the right by $A^-$. Note that $A^-$ is countable. Finally we set
$$\spt^*(\pi)=\left(\spt(\pi)\backslash (A^-\times \R)\right)\cup\bigcup_{\mu(x)>0}\{x\}\times\spt\pi_x.$$
We have subtracted countably many vertical lines from $\spt(\pi)$ so that 
\begin{align*}
\pi(\spt^*\pi)=&\int\pi_x(\{y\in\R^2|\,(x,y)\in\spt^*(\pi)\})\dd \mu(x).\\
=&\int_{\R\backslash A^-}\pi_x(\{y\in\R^2|\,(x,y)\in\spt(\pi)\})\dd \mu(x)\\
&+\sum_{x\in A^-,\,\mu(x)>0}\mu(x)\pi_x(\spt\pi_x)\\
=&\mu(\R\backslash A^-)+\sum_{x\in A^-}\mu(x)=1.
\end{align*}

Another important property is that $\spt^*(\pi)\subset \spt(\pi)$.

In Proposition \ref{car_etoile} and Theorem \ref{corstab} we will use many times Lemma \ref{help2} that relies on the following statement.
\begin{lem}\label{help}
Let $(x,y)\in\spt\pi$ where $\pi$ is a martingale transport plan and $G$ a Borel set such that $\pi(G)=1$.

If $x<y$, for any $\eps>0$ there exist $(x_1, y_1^-),(x_1,y_1^+)\in G$ with $y_1^-\leq y_1^+$, such that the point $(x_1,y^+_1)$ is in the ball of centre $(x,y)$ and radius $\eps$ and $y^-_1<x+\eps$.

If $x>y$, the symmetric statement holds as well. There exists $(x_1, y_1^-),(x_1,y_1^+)\in G$ with $y_1^-\leq y_1^+$, such that the point $(x_1,y^-_1)$ is in the ball of centre $(x,y)$ and radius $\eps$ and $y^+_1>x-\eps$.
\end{lem}
\begin{proof}
It is sufficient to prove the first statement because the second is proved in the same way. We also can assume without loss of generality that $x<y-\eps$.
We consider the usual disintegration of $\pi$ with respect to $\mu=\proj^x_\#\pi$.
Let us denote the cut of $G\cap(\{s\}\times\R)$ by $\{s\}\times G_s$. For $\mu$-almost every $s$, we have $\pi_s(G_s)=1$ and the expectation of $\pi_s$ is $s$. Moreover as $x$ is in the support of $\mu$, and $(x,y)$ in the support of $\pi$, we have also $\pi_s(]y-\eps,y+\eps[)>0$ for any $s$ in a subset $S\subset]x-\eps,x+\eps[$ of positive $\mu$-measure. As $x\notin [y-\eps,y+\eps]$, for almost every element $s\in S$, we have $\max(x-\eps,\inf\spt\pi_s)<s<\min(x+\eps,\sup\spt\pi_s)$.  Hence, we can find $(x_1,y_1^-)$ and $(x_1,y_1^+)$ in $G$ with $\max(|x-x_1|,|y-y_1^+|)\leq \eps$ and $y_1^-<x+\eps$.
\end{proof}
\begin{lem}\label{help2}
Let $\pi$ be a martingale transport plan, $(x,y)$ and $(x',y')$ elements of $\spt(\pi)$ and assume $x<x'$. If $x<y'<y$ or $x>y'>y$, then $\pi$ is not left-monotone.
\end{lem}
\begin{proof}
It is sufficient to prove the first statement. The second statement is proved in the same way. Let $(x,y)$ and $(x',y)$ be elements of $\spt(\pi)$ with $x<y'<y$ and assume by contradiction that $\pi$ is left-monotone. We introduce $\Gamma$ as in Definition \ref{lm}. The proof relies on Lemma \ref{help} where we choose $G=\Gamma$. According to this theorem for $\eps<\min(|y'-x|,|y'-y|,|x'-x|)$, there exists $x_1$ and $y^-_1,y^+_1$ with $(x_1,y^\pm_1)\in \Gamma$ such that $|x_1-x|<\eps$, $|y_1^+-y|<\eps$ and $y_1^-<x+\eps$. We are in the situation forbidden in the definition of $\Gamma$ because $x_1<x'$ and $y_1^-<y'<y^+_1$. This is not  directly a contradiction because $(x',y')$ may not be an element of $\Gamma$. Nevertheless $(x',y')\in\spt(\pi)\subset\bar{\Gamma}$ so that we can replace it with some element $(x_1',y_1')\in \Gamma$.
\end{proof}
\begin{pro}\label{car_etoile}
A martingale transport plan $\pi$ is the left-monotone coupling of $\Pi_M(\mu,\nu)$ if and only if it satisfies the following condition
\begin{itemize}
\item for every $(x,y^+)$, $(x,y^-)$ and $(x',y')$ elements of $\spt^*(\pi)$, if $x<x'$ and $y^-<y^+$, we have $y'\notin]y^-,y^+[$.
\end{itemize}
\end{pro}
\begin{proof}
Let us first prove that $\pi$ is left-monotone. The set $\spt^*\pi$ fulfils the requirements for $\Gamma$. Indeed $\pi(\spt^*\pi)=1$ and the second condition is assumed in the statement.

Conversely, we assume now that there exists some $\Gamma$ of mass $1$  that satisfies the conditions in Definition \ref{lm}. Without loss of generality, we can assume $\Gamma\subset \spt(\pi)$: just take $\Gamma\cap\spt(\pi)$. By contradiction we consider $(x,y^+)$ and $(x,y^-)$ and $(x',y')$ in $\spt^*(\pi)$ such that $x<x'$ and $y'\in]y^-,y^+[$. 
Note that $\spt^*(\pi)\subseteq \spt(\pi)\subseteq\bar{\Gamma}$. In particular each point of $\spt^*(\pi)$ can be approximated by a point of $\Gamma$.

We distinguish two cases

{\it Case 1: $y'\neq x$.} We can easily conclude applying Lemma \ref{help2} to $(x',y')$ and $(x,y^+)$ or $(x,y^-)$ depending respectively whether $x<y'$ or $x>y'$.

{\it Case 2: $y'=x$.} We can also assume $\pi(]-\infty,x[\times]x,+\infty[)=0$ because if not there exists $(x_1,y_1)\in\spt\pi$ with $x_1<x$ and $x_1<y'<y_1$, which permits us to apply Lemma \ref{help2} to $(x',y')$ and $(x_1,y_1)$ and provides a contradiction with the fact that $\pi$ is left-monotone. Hence we have $x\in A$. Remind that $(x,y^+)\in \spt\pi$ and $y^+>y'=x$. Therefore there is locally no element of $A$ on the right of $x$. Hence $x\in A^-$. According to the definition of $\spt^*(\pi)$ it implies $\pi(\{x\}\times\R)=\mu(x)>0$ and $y^-,y^+\in\spt(\pi_x)$. As $\mu(x)>0$ we must have $\pi_x(\Gamma_x)=1$ where $\Gamma_x=\{y\in\R,\,(x,y)\in\Gamma\}$. Hence we can find two points in $y_1^\pm\in\Gamma_x$ that are close to $y^\pm$. The points $(x,y_1^-)$, $(x,y^+_1)$ together with some point of $\Gamma$ close to $(x',y')$ provide a contradiction.

\end{proof}

\begin{cor}\label{car_etoile2}
A martingale transport plan $\pi$ is the left-monotone coupling of $\Pi_M(\mu,\nu)$ if and only if it satisfies the following condition
\begin{itemize}
\item for every $(x^-,y^+)$, $(x^+,y^-)$ and $(x',y')$ elements of $\spt^*(\pi)$, if $x^-\leq x^+<x'$ and $y^-<y^+$, we have $y'\notin]y^-,y^+[$.
\end{itemize}
\end{cor}
\begin{proof}
If $\pi$ satisfies this condition, it also satisfies the sufficient condition in Proposition \ref{car_etoile} so that it is the left-curtain coupling. For the other inclusion, let us consider $\pi$ and three points $(x^-,y^+)$, $(x^+,y^-)$ and $(x',y')$ in $\spt^*(\pi)$ as in the statement but such that $y'\in]y^-,y^+[$. We prove that it is not a left-curtain coupling. If $x^-=x^+$, we simply use the necessary condition in Proposition \ref{car_etoile}. Hence we assume $x^-<x^+$. We can now apply Lemma \ref{help2} with $(x',y')$ and $(x^-,y^+)$ or $(x^+,y^-)$ depending on whether $x^-<y'$ or $x^+>y'$.
\end{proof}

With the last statement we can now implement the strategy of Theorem 5.20 in \cite{Vi2} in order to prove the continuity of the curtain coupling $\curt$. Actually with Corollary \ref{car_etoile2} it is by now possible to consider triples of points that are typical for the measure $\pi^{\otimes 3}$ instead of vectors $(x,y^-,y^+,x',y')$ in $\R^5$.

Recall that in paragraph \ref{quanti} we will prove the Lipschitz continuity of $\curt$ for a specific semimetric $Z$ by using another method.
\begin{them}\label{corstab}
We consider the mapping $\curt:(\mu,\nu)\in \mathcal{D}_{\leqc} \mapsto \pi_\lc$, where $\mathcal{D}_{\leqc}=\{(\mu,\nu)\in \p^2: \mu\leqc \nu\}$. This mapping is continuous from $\mathcal{D}_{\leqc}$ to $\p(\R^2)$.
\end{them}

Theorem \ref{corstab} is proven for the usual weak convergence of probability measures in $\p$ as well as in $\p(\R^2)$ for the range. Nevertheless, according to Remark \ref{rem_topo}, one may also consider the two reinforced topologies (with finite first moment) induced by $W$ and $W^{\R^2}$


\begin{proof}[Proof of Theorem \ref{corstab}]
Let us introduce a sequence $(\mu_n,\nu_n)$ converging to $(\mu,\nu)$. We assume that for every $n\in\N$, $\pi_n$ is a left-monotone coupling of $\mu_n$ and $\nu_n$. We will prove that $\pi_n$ has a limit $\pi$ and that it is also left-monotone. Due to Prokhorov's theorem on compactness and the uniqueness of a left-monotone martingale coupling with given marginals, we can reduce the proof to the case we know that $\pi_n$ converges to $\pi$.

We introduce the set $E=\{(x^-,y^+,x^+,y^-,x',y')\in\R^6|\,x^-\leq x^+<x'\text{ and }y^-<y'<y^+\}$. Assume that there is a vector $v\in E\cap (\spt^*(\pi))^3$, which according to Corollary \ref{car_etoile2} is equivalent to the fact that $\pi$ is not left-monotone. We will see that it implies that some $\pi_n$ is not left-monotone. Before we proceed to the proof, let us stress that $\pi_n^{\otimes 3}$ converges to $\pi^{\otimes 3}$ and as $v\in (\spt(\pi^{\otimes 3}))=(\spt(\pi))^3$ we obtain a sequence $v_n$ with $v_n\in (\spt^*(\pi_n))^3$ and $v_n\to v$. Our goal will be to prove $v_n\in E$ or directly that $\pi_n$ is not left-monotone.

We distinguish two main cases.

\emph{Case 1: $v\in E^\circ$.} As $E^\circ$ is an open set, $v_n\in E^\circ$ for $n$ sufficiently large, which provides the contradiction with the fact that $\pi_n$ is left-monotone.


\emph{Case 2: $v\in\partial E$.} We have $x^-=x^+$ and denote this real number simply by $x$. The arguments for the different subcases that we will distinguish are very similar to the ones in the proof of Proposition \ref{car_etoile}. The cases 2.1 and 2.2 corresponds to Case 1 and Case 2 of this proposition. 

Recall that $v_n=(x^-_n, y^+_n,x^+_n,y^-_n,x'_n,y'_n)\in\spt^*(\pi_n)^3$ tends to $v$. If $x_n^-\leq x_n^+$ we are done because with Corollary \ref{car_etoile2} this implies $v_n\in E$. Hence we must assume $x_n^-> x_n^+$. 

{\it Case 2.1: $y'\neq x$.} This is not possible. If for instance $x<y'$, the relations $x^-_n<x'_n$ and $<x^-_n<y_n'<y^+_n$ hold if $n$ is sufficiently large so that we can use Lemma \ref{help2} for the points $(x_n,y^+_n),\,(x'_n,y'_n)\in\spt(\pi_n)$. Hence we contradict that $\pi_n$ is left-monotone.

{\it Case 2.2.} Hence up to now we have assumed $x=x^-=x^+$ and $x_n^->x_n^+$ and $y'=x$. We show now that $x\in A(\pi)$. Indeed if it is not true there exists a sequence $(s_n,t_n)\in\spt(\pi_n)$ converging to $(s,t)\in\spt\pi\cap(]-\infty,x[\times]x,+\infty[)$. Thus, if $n$ is sufficiently large, recalling that $y'_n$ tends to $x$ we can apply Lemma \ref{help2} to $(s_n,t_n)$ and $(x'_n,y'_n)$. Indeed we have $s_t<x'_n$ and $s_n<y'_n<t_n$.

Let us see that $x\in A\backslash A^-$ is impossible. Actually $(x,y^+)\in\spt^*\pi$ is an element of $\spt\pi$ and $x<y^+$. It follows $\pi(]-\infty,x_1[\times]x_1,+\infty[)>0$ for any $x_1\in]x,y^+[$. Hence $x\in A^-$. It follows $\mu(x)>0$.

We assume for simplicity that $x=0$. We denote $\mu(x)\cdot\pi_x(]y^+/2,+\infty[)$ by $m$. It is not zero because $y^+\in \spt(\pi_x)$. Let $\eps>0$ be strictly smaller than $\min(x',y^+/2)$. We also assume that it is sufficiently small to satisfy $a=(y^+/2-\eps)(\pi_x(]y^+/2,+\infty[)/8)-\eps>\eps$ and $\mu([-\eps,\eps])<2\mu(x)$. We know that
\begin{align*}
\liminf&\pi_n(]-\eps,+\eps[\times]y^+/2,+\infty[)\\
\geq &\pi(]-\eps,+\eps[\times]y^+/2,+\infty[)\\
\geq &\mu(x)\cdot\pi_x(]y^+/2,+\infty[)=m.
\end{align*}
and for $\mu_n=\proj^x_\#\pi_n$,
$$\limsup\mu_n([-\eps,\eps])\leq\mu([-\eps,\eps])$$
Hence, there is $n$ such that $\pi_n(]-\eps,+\eps[\times]y^+/2,+\infty[)>  m/2$ and $\mu_n([-\eps,\eps])< 2\mu([-\eps,\eps])<4\mu(x)$. Therefore on a set $B\subseteq]-\eps,\eps[$ of positive $\mu_n|_{]-\eps,+\eps[}(\dd s)$ the measure $(\pi_n)_s(]y^+/2,+\infty[)$ is greater than $\mu_n(]-\eps,\eps[)^{-1}\cdot m/2>\pi_x(]y^+/2,+\infty[)/8$. Hence for $s\in B$, using the fact that the barycenter of $(\pi_n)_s$ is $s$, we have 
$$(\pi_n)_s(]-\infty,-a[)>0$$
where $a=(y^+/2-s)(\pi_x(]y^+/2,+\infty[)/8)-s$. Remind that $-a<-\eps<s$. Let $\Gamma$ be a Borelian set of $\R^2$ such that $\pi_n(\Gamma)=1$. As for almost every $s$, we have $(\pi_n)_s(\{t\in\R|\,(s,t)\in\Gamma)>0\})$, we obtain that there are $(s,t^-)$ and $(s,t^+)$ in $\Gamma$, with $t^+>y^+/2$ and $t^-<-\eps$, and $(s',t')\in \Gamma$ close to $(x',0)$ such that $t'\in]t^-,t^+[$. We conclude with Definition \ref{lm} that $\pi_n$ is not be left-monotone, which contradicts our assumptions.
\end{proof}

\begin{rem}%
Theorem \ref{corstab} provides a more direct and intuitive introduction of $\pi_\lc=\curt(\mu,\nu)$ than Definition \ref{lc}. In this alternative presentation relying on \cite[Section 2]{BJ} (see also Lemma \ref{ahyo}) one considers a sequence of atomic measures $\mu_n$ that converges to $\mu$ (see for instance point 3 in the proof of Proposition \ref{pro_un}). We may assume $\mu_n\leqc \mu$ in order to satisfy $(\mu_n,\nu)\in\mathcal{D}_{\leqc}$. The left-curtain couplings $\pi_n=\curt(\mu_n,\nu)$ can be described easily, as is done for instance in the proof of Lemma \ref{ahyo}. For that purpose it is not necessary to introduce the shadows in full generality but only to know what is the shadow of an atom. According to the theory $\pi_\lc$ is the limit of $(\pi_n)_n$.

Note that without the theory from \cite{BJ}, Theorem \ref{corstab} can only prove that the accumulation points of the sequence $(\pi_n)_n$ are all left-monotone couplings. Without Proposition \ref{synthese} it is not known that the left-monotone elements of  $\M(\mu,\nu)$ are reduced to $\{\pi_\lc\}$. Hence the alternative presentation explained in the present remark can not be seen as a definition. 
\end{rem}

We end the paragraph on qualitative continuity with two results on the continuity of the shadows that will be useful in section \ref{quanti}.

\begin{lem}[Role of  the mass of $\nu$ close to $\pm\infty$]\label{charge_a_droite}
Let  $\mu$ and $\nu$ be measures of $\m$ such that $\mu\leqcp \nu$. Let $(\nu_n)_n$ such that $\inf(\spt\nu_n)$ tends to $+\infty$.
The sequence $S^{\nu+\nu_n}(\mu)$ tends to $S^\nu(\mu)$ in $\m$.

The similar statements hold if $\sup(\spt\nu_n)$ tends to $-\infty$ or $\nu_n([a_n,b_n])=0$ with $-a_n,\, b_n\to+\infty$.
\end{lem}
\begin{proof}
1. We will prove that the potential function of $S^{\nu+\nu_n}(\mu)$ pointwise converges to the potential function of $S^\nu(\mu)$. Remind that it was introduced after $W$ at the beginning of Chapter \ref{premier}. Fix $a\in \R$ and $\eps>0$ and let $\delta>0$ be such that any measure $\alpha\leqp \nu$ of mass $\alpha(\R)\leq \delta$ satisfies $\int|x-a|\dd \alpha\leq \eps$. Let $\mu'$ be the leftmost measure smaller than $\mu$ and  of mass $\mu'(\R)=\mu(\R)-\delta$. 

2. It is enough to prove that for any $n$ satisfying $\inf\spt\nu_n> \sup [\spt S^\nu(\mu')]$ we have 
\begin{align}\label{fin}
|u_{S^\nu(\mu)}(a)-u_{S^{\nu+\nu_n}(\mu)}(a)|=\left|\int |x-a| \dd S^\nu(\mu)-\int |x-a| \dd S^{\nu+\nu_n}(\mu)\right|\leq \eps.
\end{align}
Before we state this inequality, we have to prove that $\sup [\spt S^\nu(\mu')]$ is finite. Actually, the shadow of $\mu'$ restricted on $]\sup\spt\mu',+\infty[$ (more precisely $(S^\nu(\mu')-\mu')_+|_{[\sup(\spt \mu'),+\infty][}$) is made of the leftmost quantiles of $\nu|_{]\sup(\spt\mu'),+\infty[}$ as is proved in Lemma \ref{ombre_a_droite}. As some mass must remain for the shadow of $\mu-\mu'$, as explained in Proposition \ref{shadowsum}, this can not be the full $\nu|_{]\sup(\spt\mu'),+\infty[}$. 

3. As $S^\nu(\mu')\leqp \nu+\nu_n$ and $\mu'\leqc S^\nu(\mu')$, we have $S^{\nu+\nu_n}(\mu')\leqc S^\nu(\mu')$. Considering Strassen's Theorem (Proposition \ref{strassen1}), we obtain the corollary that $\sup(\spt S^{\nu+\nu_n}(\mu'))\leq \sup(\spt S^\nu(\mu'))$. With the hypothesis made in 2. on the support of $\nu_n$ this proves $S^{\nu+\nu_n}(\mu')\leqp \nu$. Hence $S^{\nu+\nu_n}(\mu')\geqc S^\nu(\mu')$. Finally $S^{\nu+\nu_n}(\mu')= S^\nu(\mu')$.

4. We denote by $\sigma$ the latter measure. Applying Proposition \ref{shadowsum} to the shadow of the sum $\mu'+(\mu-\mu')$ we get
\begin{align*}
\int |x-a| \dd S^{\nu+\nu_n}(\mu)=\int |x-a| \dd\sigma+\int |x-a| \dd S^{\nu+\nu_n-\sigma}(\mu-\mu')
\end{align*}
and
\begin{align*}
\int |x-a| \dd S^{\nu}(\mu)&=\int |x-a| \dd\sigma+\int |x-a| \dd S^{\nu-\sigma}(\mu-\mu').\\
\end{align*}
As the shadow of $\mu-\mu'$ in $\nu+\nu_n-\sigma$ is smaller in the convex order than its shadow in $\nu-\sigma$ and reminding the choices made in 1. we get
\begin{align*}
0\leq\int |x-a| \dd S^{\nu+\nu_n-\sigma}(\mu-\mu')\leq\int |x-a| \dd S^{\nu-\sigma}(\mu-\mu')\leq\eps,
\end{align*}
so that \eqref{fin} is established.

5. In the case $\sup(\spt\nu_n)$ tends to $-\infty$ we just do the symmetric proof. If $\nu_n([a_n,b_n])=0$ with $-a_n,\, b_n\to+\infty$, we implement a similar proof with the following modification: at step 1. $\mu'$ is chosen in the middle of $\mu$ so that its shadow has a compact support (adapt the argument in 2. that relies on Lemma \ref{ombre_a_droite}). Steps 3. and 4. go in the same way. 
\end{proof}

With Theorem \ref{corstab} we obtain the corollary. 
\begin{cor}
Under one of the three hypotheses of Lemma \ref{charge_a_droite}, we have
$\curt(\mu,S^{\nu+\nu_n}(\mu))\stackrel{\m(\R^2)}{\longrightarrow} \curt(\mu,S^\nu(\mu))$.
\end{cor}

We remind another result of stability from \cite{BJ}.
\begin{pro}[Proposition 4.15 in \cite{BJ}]\label{limit_measure}
Assume that $(\mu_n)_n$ is increasing in the convex order and $\mu_n\leqc\mu\leqcp \nu$ for every $n\in\N$. Then both $(\mu_n)_n$ and $(S^\nu(\mu_n))_n$ converge in $\m$. If we call $\mu_\infty$, respectively $S_\infty$ the limits, then the measure $S_\infty$ is the shadow of $\mu_\infty$ in $\nu$. 
\end{pro}

Again with Theorem \ref{corstab} we obtain a corollary.
\begin{cor}
Under the hypotheses and notations of Proposition \ref{limit_measure}, we have
$\curt(\mu_n,S^{\nu}(\mu_n))\stackrel{\m(\R^2)}{\longrightarrow} \curt(\mu_\infty,S_\infty)$.
\end{cor}

\subsection{Lipschitz continuity of the shadow, quantitative estimates.}\label{quanti}

In this section we give a quantitative version of Theorem \ref{corstab} by using other methods. We start with definitions.
\subsubsection{Top and down measures}
\begin{defi}\label{defi_top}
Let $\mu$ and $\nu$ satisfy $\mu(\R)=\nu(\R)$ and call $t$ this constant. We define the top and the down measures of $\mu$ and $\nu$ as
$$\Top(\mu,\nu)=\max(G_\mu,G_\nu)_\#\lambda_{[0,t]}\quad\text{and}\quad\down(\mu,\nu)=\min(G_\mu,G_\nu)_\#\lambda_{[0,t]}.$$\end{defi}

\begin{ex}\label{topatom}
i) If $\mu$ and $\nu$ are probability measures and $\law(X,Y)$ is the quantile coupling of these measures, $\Top(\mu,\nu)$ and $\down(\mu,\nu)$ are simply the laws of $\max(X,Y)$ and $\min(X,Y)$.

ii) If $\mu=\sum_{i=1}^n \delta_{x_i}$ and $\nu=\sum_{i=1}^n \delta_{y_i}$ with $(x_i)_i,\,(y_i)_i$ non decreasing sequences, then $\Top(\mu,\nu)=\sum_{i=1}^n \delta_{\max(x_i,y_i)}$ and $\down(\mu,\nu)=\sum_{i=1}^n \delta_{\min(x_i,y_i)}$.
\end{ex}
\begin{lem}\label{topkanto}
Let $\mu$ and $\nu$ be as in Definition \ref{defi_top}. We have
\begin{align*}
W(\mu,\nu)&=W(\mu,\down(\mu,\nu))+W(\down(\mu,\nu),\nu).\\
&=W(\mu,\Top(\mu,\nu))+W(\Top(\mu,\nu),\nu).
\end{align*}
\end{lem}
\begin{proof}
The proof simply relies on the formula $W(\mu,\nu)=\int|G_\nu-G_\mu|\dd \lambda$ in Lemma \ref{stokanto} and $|G_\nu-G_\mu|=(G_\mu-G)+(G_\nu-G)=(G'-G_\mu)+(G'-G_\nu)$ where $G=\min(G_\mu,G_\nu)$ and $G'=\max(G_\mu,G_\nu)$.
\end{proof}
\begin{lem}\label{Top_com}
$\Top(\mu,\mu')+\nu=\Top(\mu+\nu,\mu'+\nu)$
\end{lem}
\begin{proof}
One can check with the definition of $\Top$ by random variables that $F_{\Top(\mu,\mu')}(t)=\min(F_\mu(t),F_{\mu'}(t))$. Hence
for every $t\in\R$.
\begin{align*}
F_{\Top(\mu,\mu')+\nu}(t)&=\min(F_\mu(t),F_{\mu'}(t))+F_\nu(t)\\
&=\min(F_\mu(t)+F_\nu(t),F_{\mu'}(t)+F_\nu(t)).
\end{align*}
But this is also the cumulative distribution function of $\Top(\mu+\nu,\mu'+\nu)$. If $\mu,\mu'$ are not probability measures, the result follows from their normalisation.
\end{proof}

The shadow $S^\nu(\mu)$ is only defined for $\mu\leqcp\nu$, which may be restrictive for some proofs. But if, roughly speaking, we add mass close to infinity to $\nu$ this shadow can exist. The two next lemmas permit us to implement this idea, which for instance plays a role in the final proof of Section \ref{lips_esti}.

\begin{lem}[adding mass at $\pm\infty$]\label{add_queue}
If $\mu\leqcps \nu$, for any $a\in \R$ there exists $\nu'$ with $\nu'$ concentrated on $]-\infty,a]$ and $\mu\leqcp \nu+\nu'$.

Similarly, if no assumption is done on $\mu$, for any $a,\,b\in\R$, there exists $\nu'$ with $\nu'(]a,b[)=0$ and $\mu\leqcp \nu'$.
\end{lem}
\begin{proof}
1. We start to prove the first result in the special case $\mu\leqps \nu$. For any $a\in \R$, we will find $\nu'$ with $\nu'(]a,+\infty[)=0$ and $\mu'\leqcp \nu+\nu'$. Let assume without loss of generality that $\mu$ is a probability measure and applying Theorem \ref{echelle} let $X\leq Y$ be random variables with laws $\mu=\law(X)$ and $\law(Y)\leqp \nu$. Let us fix $a\in\R$ and consider in a first time the case $\mu(]-\infty,a])=0$. In this case we introduce $U$ a random variable independent from $X$ and define $Z$
\begin{align*}
Z=\begin{cases}
Y&\text{if }U\geq\frac{Y-X}{Y-a}\\
a&\text{otherwise.}
\end{cases}
\end{align*}
Therefore $\E(Z\mid X)=X$. Observe that $\law(Z)\leqp \nu+\P(Z=a)\delta_a$ so that for $\nu'=\delta_a$ we have $\mu\leqcp\nu+\nu'$.

In the second case we can write $\mu=\mu_1+\mu_2$ where $\mu_1$ is concentrated on $A=]-\infty,a]$ and $\mu_2(A)=0$. By using the result just proved for $\mu_2$, we obtain $\nu'_2$ concentrated on $A$ such that $\mu_2\leqcp\nu+\nu'_2$. We set $\nu'_1=\mu_1$ so that we also have $\mu_1\leqcp \nu'_1$. As $\nu'=\nu'_1+\nu'_2$ is concentrated on $A$, we are done.

2. If $\mu\leqcps \nu$, According to Theorem \ref{echelle} there exists $\mu'$ such that $\mu\leqc \mu'$ and $\mu'\leqps \nu$. Therefore with part 1. we find for every $a\in \R$ a measure $\nu'$ concentrated on $]-\infty,a]$ such that $\mu'\leqcp \nu+\nu'$. But $\mu\leqcp\mu'$ so that $\mu\leqe \nu+\nu'$ is also satisfied.

3. For the second statement fix $a$ and $b$ and set $\nu=\mu(]-\infty,b])\delta_b+\mu|_{]b,\infty[}$. As $\mu\leqs \nu$, one can apply part 1. of the present proof.
\end{proof}

\subsubsection{Semimetric $Z$ on the space of transport plans.}
\begin{defi}\label{defz}
Let $\pi$ and $\pi'$ be two transport plans.
We define $Z(\pi,\pi')$ as
\begin{align*}
Z(\pi,\pi')&=\inf_{(\pi^s,\,\pi'^s)_{s\in[0,1]}}\sup_{s\in[0,1]} \max\left(W(\mu^s,\mu'^s), W(\nu^s,\nu'^s)\right).\\
&=\max\left(W(\mu,\mu')\,;\inf_{(\pi^s,\,\pi'^s)_{s\in[0,1]}}\sup_{s\in[0,1]}W(\nu^s,\nu'^s)\right)
\end{align*}
where $\mu^s,\,\nu^s$ are the marginals of $\pi^s$ and $\mu'^s,\,\nu'^s$ the marginals of $\pi'^s$. The infimum is taken among all the families $(\pi^s)_{s\in[0,1]},\,(\pi'^s)_{s\in[0,1]}$ that satisfy
\begin{enumerate}
\item $\forall s\in[0,1],\,\pi^s(\R^2)=s$,
\item $\forall s\in[0,1],\,\exists x\in [-\infty,+\infty],\,\pi_{]-\infty,x[\times \R} \leqp\pi^s\leqp \pi_{]-\infty,x]\times \R}$ (in fact $x=G_\mu(s)$),
\item if $s\leq t$, then $\pi^s\leqp \pi^t$.
\end{enumerate}
\end{defi}
\begin{rem}\label{first}
If the first marginals of $\pi$ and $\pi'$ are continuous, there is no freedom in the choice of $(\pi^s)_{s\in[0,1]},\,(\pi'^s)_{s\in[0,1]}$. Hence $Z$ can be reformulated as
\begin{align*}
Z(\pi,\pi')=\max
\left\{
\begin{aligned}
&W(\mu,\mu')\\
&\sup_{F_\mu(x)=F_{\mu'}(x')}W(\proj^y_\#\pi|_{]-\infty,x]\times \R},\proj^y_\#\pi'|_{]-\infty,x']\times \R})
\end{aligned}
\right.
\end{align*}
\end{rem}
\begin{pro}\label{semimetric}
$Z$ is a semimetric on $\p(\R^2)$ and the triangle inequality is satisfied on the subspace of measures with continuous first marginal. Moreover if $Z(\pi_n,\pi)\to 0$ we have $\pi_ n\to\pi$ for the topology $\mathcal{T}_1(\R^2)$.

\end{pro}
\begin{proof}
1. It is clear that $Z$ is symmetric and $Z(\pi,\pi')=0$ if and only if $\pi=\pi'$. This principle is used in the definition of the left-curtain coupling (see Definition \ref{lc}): the measures $\proj^y_\#\pi|_{]-\infty,x]\times \R}$ completely determine $\pi$.
The triangle inequality on the subspace of measures with continuous first marginal follows from the triangle inequality of $W$ and Remark \ref{first}.

2. Assume that $Z(\pi_n,\pi)$ tends to $0$ and for every $n$ introduce families $(\pi^s_n)_{s\in[0,1]}$ that together with some families decomposing $\pi$ satisfy the limit condition. We introduce a sequence $\pi'_n$ of transport plans defined as intermediate measure between the measures $\pi_n$ and $\pi$. Their first marginals are $\proj^x_\#\pi'_n=\mu$ like $\pi$ and one associates them with the same family $\nu_n^s=\proj^y_\#\pi_n^s$ like $\pi_n$. More precisely, let $(U,X,(X_n,Y_n))_{n\in\N}$ be such that  
\begin{itemize}
\item $\law(U)=\lambda_{[0,1]}$
\item $(U,X,X_n)$ is comonotonic (i.e. $X=G_X(U)$ and $X_n=G_{X_n}(U)$ almost surely),
\item $\law(X)=\mu$,
\item $\law(X_n,Y_n)=\pi_n$,
\item $\nu_n^s=s\law(Y_n\mid U\leq s)$,
\end{itemize}
then $\pi_n'$ is the law of $(X,Y_n)$ and $\pi_n=\law(X_n,Y_n)$.

For any $s\in[0,1]$, let $\nu^s$ be $\proj^y_\#\pi^s$ with $(\pi^s)_{s\in[0,1]}$ some family admissible in the sense of Definition \ref{defz} (1)--(3). We do not necessarily assume $W(\nu^s_n,\nu^s)\to 0$ for every $s$ but we will need the following remark in our proof: for any $x\in\R$, the measure $\nu^{F_\mu(x)}$ is completely characterised and it is $\proj^y_\#(\pi|_{]-\infty,F_\mu(x)]\times \R})$. Indeed this is the only measure of mass $F_\mu(x)$ satisfying point (2) in Definition \ref{defz}. Hence for any $x\in \R$, $W(\nu_n^{F_\mu(x)},\nu^{F_\mu(x)})\to 0$ because the family $(\nu_n^s)_{s\in [0,1]}$ corresponding to $(\pi_n^s)_{s\in [0,1]}$ has been properly chosen for the convergence.

We can now proceed to the proof. We want to prove $W^{\R^2}(\pi_n,\pi'_n)\to 0$ and $W^{\R^2}(\pi'_n,\pi)\to 0$. On the one hand we have $W^{\R^2}(\pi_n,\pi'_n)\leq \E(|X-X_n|+|Y_n-Y_n|)=W(\mu_n,\mu)\leq Z(\pi_n,\pi)\to 0$. On the other hand, we prove that at any continuity point $(x,y)\in \R^2$ of $F_\pi$, the sequence $F_{\pi_n}$ pointwise converges to $F_\pi$. According to a classical characterisation (see e.g. \cite[Example 2.3]{Bi}), this will imply $\pi_n\to\pi$ in the weak topology $\mathcal{T}_{\mathrm{cb}}(\R^2)$. Moreover, as $W(\nu,\nu_n)=W(\nu^1,\nu_n^1)\to 0$ and $\mu$ is the first marginal of both $\pi'_n$ and $\pi$, one can apply Remark \ref{rem_topo}. Therefore the weak convergence $\pi_n'\to\pi$ implies the weak convergence with finite first moment $W^{\R^2}(\pi'_n,\pi)\to 0$.

Let $(x,y)$ be a continuity point of $F_\pi$. We have $F_\pi(x,y)=\nu^{F_\mu(x)}(]-\infty,y])$ and also $F_{\pi'_n}(x,y)=\pi'_n(]-\infty,x]\times]-\infty,y])=\nu_n^{F_\mu(x)}(]-\infty,y])$. This tends to $\nu^{F_\mu(x)}(]-\infty,y])$ because $W(\nu^{F_\mu(x)},\nu_n^{F_\mu(x)})\to 0$ and $y$ is a continuity point of $\nu^{F_\mu(x)}$. We conclude that $\pi'_n$ weakly converges to $\pi$.  Therefore
$$W^{\R^2}(\pi_n,\pi)\leq W^{\R^2}(\pi_n,\pi'_n)+W^{\R^2}(\pi'_n,\pi)\to 0.$$
\end{proof}

The following estimate is one of our main theorems. It provides a quantitative estimate on the Lipschitz continuity of the shadow projection $(\mu,\nu)\mapsto S^\nu(\mu)$.
\begin{them}\label{lips}
Let $\mu,\mu'$ and $\nu,\nu'$ be elements $\m$. We assume $\mu\leqcp \nu$ and $\mu'\leqcp \nu'$ respectively. We assume also
$\mu(\R)=\mu'(\R)$ and $\nu(\R)=\nu'(\R)$. The following relation holds
$$W(S^\nu(\mu),S^{\nu'}(\mu'))\leq W(\mu,\mu')+2W(\nu,\nu')$$
\end{them}

The proof of the theorem is postponed at the end of the section. It relies on all results in between including Proposition \ref{pro_un} and Proposition \ref{pro_deux}. The first proposition is concerned with $\nu=\nu'$ and the second with $\mu=\mu'$. Before we start with this program let us state a corollary of Theorem \ref{lips} that gives a quantitative turn to Theorem \ref{corstab} in terms of the semimetric $Z$. A similar result can not be satisfied with $W^{\R^2}$ in place of $Z$ as is explained later in Example \ref{Wass_plouf}

\begin{cor}\label{doubidou}
Consider the mapping $\curt:(\mu,\nu)\in \mathcal{D}_{\leqc} \mapsto \pi_\lc$, where $\mathcal{D}_{\leqc}=\{(\mu,\nu)\in \p^2: \mu\leqc \nu\}$. This mapping is continuous from $\mathcal{D}_{\leqc}$ to $\M$ equipped with the topology $\mathcal{T}_1(\R^2)$. More precisely
$$Z(\curt(\mu,\nu),\curt(\mu',\nu'))\leq W(\mu,\mu')+2W(\nu,\nu')$$
\end{cor}
\begin{proof}
Like in paragraph \ref{plusgrand} we have $\mu=(G_\mu)_\#\Lg_{[0,1]}$. We introduce $\mu^s=(G_\mu)_\#\Lg_{[0,s]}$ and $\nu^s=S^\nu(\mu^s)$. In a similar way to Definition \ref{lc} of $\pi_\lc$, we obtain a unique family $\pi^s$, increasing for $\leqp$, and with marginals $\mu^s$ and $\nu^s$. We proceed in the same way for $\mu'$ and $\nu'$. We obtain the wanted estimate by applying Theorem \ref{lips} to these measures.
\end{proof}

We start with the preliminaries of the proof of Theorem \ref{lips}.
\subsubsection{Variations in $\mu$.}\label{mu}
\begin{lem}[Important lemma]\label{ahyo}
Let $n$ an integer and $\mu,\mu'$ be two measures that are the sum of $n$ atoms of the same mass and such that $\mu\leqs \mu'$. If $\nu\in\m$ satisfies $\mu\leqe \nu$ and $\mu'\leqe \nu$ we have $S^\nu(\mu)\leqs S^{\nu}(\mu')$.
\end{lem}

\begin{proof}
Without loss of generality, we can assume that the atoms are all Dirac masses (the mass is 1). We write $\mu=\sum_{i=1}^n\delta_{x_i}$ with $x_{i}\leq x_{i+1}$ for any $i<n$ and use the same notations for $\mu'$. As $\mu\leqs \mu'$, one has $x_i\leq x'_i$ for any $i\leq n$.

The proof relies on the description of the shadow of a measure concentrated in one point as $G_\#\Lg_{]q,q+\alpha]}$ where $G$ is the inverse cumulative function of the target measure $G_\nu$ and $\alpha$ is the mass of the atom (see Example \ref{un_atome}). It relies also on the decomposition Proposition \ref{shadowsum} as in Example \ref{more_atoms}: if $\bar{\mu}$ (resp. $\bar{\mu}'$) is the restriction of $\mu$ (resp. $\mu'$) to $\{x_1,\ldots,x_{n-1}\}$ (resp. $\{x'_1,\ldots,x'_{n-1}\}$) we have
$$S^\nu(\mu)=S^\nu(\bar\mu)+S^{\nu-S^\nu(\bar\mu)}(\delta_{x_{n}}).$$
and the similar equation holds for $\mu',\bar\mu'$ and $\delta_{x'_n}$.

We will prove the result by induction on $n$, the number of atoms, not greater than $m=\nu(\R)$. For $n=1$, the statement is obvious. Actually, denoting by $G$ the inverse cumulative function of $\nu$ (it satisfies $G_\#\Lg_{]0,m]}=\nu$) the shadow measures can be written as $G_\#\lambda_{]p,p+1]}$ and $G_\#\lambda_{]p',p'+1]}$ where $p\leq p'$.

If $n\geq 2$ we adopt the notations $\bar{\mu},\,\bar{\mu}'$ introduced above and we assume a statement stronger than the lemma that we call $\mathcal{H}_{n-1}$: there exists two sets $\bar J,\bar J'\subset ]0,m]$ that satisfy
\begin{itemize}
\item the masses of $\bar{J}$ and $\bar{J'}$ are $n-1$,
\item $\bar{J}$ is a disjoint union of intervals of type $]a,b]$. The same holds for $\bar{J'}$,
\item these intervals have integer length,
\item  $\Lg_{\bar J}\leqs\Lg_{\bar J'}$ (during this proof we note it simply $\bar J\leqs \bar{J'}$). In particular, $\max(\bar J)\leq \max(\bar{J'})$,
\item $G_\#\Lg_{\bar J}=S^\nu(\bar{\mu})$ and $G_\#\Lg_{\bar{J'}}=S^\nu(\bar{\mu}')$ (in particular $S^\nu(\bar\mu)\leqs S^\nu(\bar{\mu}')$).
\end{itemize}
As $\nu$ may possess atoms of mass $>1$, the sets $\bar J$ and $\bar J'$ may not be unique. We assume that $\max\bar{J}$ and $\max\bar{J'}$ are as small as possible: no other set  satisfying the five conditions before has a smaller maximum. Note also that we proved $\mathcal{H}_1$ in the paragraph above.

Starting from $\bar J$ and $\bar{J'}$ as in $\mathcal{H}_{n-1}$ where $n\leq m$,  we now construct sets $J\supset \bar J$ and $J'\supset \bar J'$ that satisfy $\mathcal{H}_n$ where $\mu,\,\mu'$ replace $\bar\mu,\,\bar\mu'$.  
In fact we follow the way described at the beginning of the proof and look for the shadow of $\delta_{x_n}$   
(resp. $\delta_{x'_n}$) in $\nu-S^\nu(\bar\mu)=G_\#\Lg_{]0,m[\setminus \bar J}$ (resp. $\nu-S^\nu(\bar\mu')$). We obtain the restriction of these measures to a ``quantile interval'', which can be described as $G_\#\Lg_{I\setminus \bar J}$ and $G_\#\Lg_{I'\setminus \bar J'}$, where $I=]p,q]$ and $I'=]p',q']$ are intervals and $I\setminus \bar J$ and $I'\setminus \bar J'$ are pseudo-intervals (see Example \ref{more_atoms}) of Lebesgue measure $1$. 
If the shadow is a Dirac mass, several choices of $I$ may be available. We choose the smallest possible $\max I$, respectively $\max I'$. Finally the sets $J=\bar J\cup I$ and $J'$ are the union of intervals of integer length. Moreover $S^\nu(\mu)=G_\#\Lg_J$ and $S^\nu(\mu')=G_\#\Lg_{J'}$ as we want.

We still have to prove the relation $J\leqs J'$. Our first step is to see $\max I> \max \bar J$ (and $\max I'>\max \bar{J'}$, which can be shown in the same way). Indeed it is clear if $x_n> x_{n-1}$. If $x_n=x_{n-1}$, a problem may happen if the shadows of both $\delta_{x_{n-1}}$ and $\delta_{x_n}$ are $\delta_{x_n}$, but we recall that $\max \bar J$ was the smallest possible value coherent with $\mathcal{H}_{n-1}$ so that $\max I> \max \bar J$ as we want. Therefore $q:=\max I$ and $\bar{J}$ completely determine $J$. Indeed, one obtains $J$ in adding the greatest real numbers in $]0,m]\setminus \bar{J}$ that are smaller than $q$. One proceeds until the set has measure $1$. The result can also be written $J=\bar{J}\cup]p,q]$.

One can see the barycenter of $\mu$ as a continuous and increasing function of $q$ and the same is true for $\mu'$ and $q'=\sup{J'}$. Let us fix $q'$ and consider $q$ as a variable. As $\bar{J}\leqs \bar{J}'$, one has also $(]0,m]\setminus \bar J')\leqs (]0,m]\setminus \bar J)$. Hence if for the shadows of $x_n$ and $x'_n$ in $\nu-S^\nu(\bar{\mu})$ and $\nu-S^\nu(\bar{\mu}')$ respectively, we start from the same value $q=q'>\max \bar{J'}\geq \max \bar J$ at the right of the interval $]0,m]$ and collect to the left the real numbers in $]0,m]\setminus\bar{J}$ and $]0,m]\setminus\bar{J}'$ respectively until one has a set of mass $1$, the set that we obtain for $S^{\nu-S^\nu(\mu)}(\delta_{x_n})$ is stochastically greater as the one for $S^{\nu-S^\nu(\bar{\mu}')}(\delta_{x'_n})$. In other words $I'\setminus \bar J'\leqs I\setminus \bar J$. This relation still holds after the push-forward $G_\#$. Taking the barycenters, we obtain $x'_{n}\leq  x_n$. But the hypothesis of the lemma states $x_{n}\geq  x'_n$. Having in mind the continuity and the monotonicity of $x_n$ with respect to $q$ we see that the correct position of $q$ satisfies $q\leq q'$.

The length of  the rightmost interval of $J'$ is an integer that we denote by $k$. As $q'\geq q$ the upper part of mass $k$ of $S^\nu(\mu')$ in the stochastic order is greater, for the same order, than the corresponding measure part of $S^\nu(\mu)$. The rest of $J'$ is included in $\bar{J}'$. Due to the induction, it is greater than the corresponding left part of $\bar{J}$ of mass $n-k$. This left part of $\bar{J}$ is greater (in the stochastic order) than the left part of $J$, that is the most left part of mass $n-k$, because $\bar{J}\subseteq J$. Thus $J\leqs J'$ and this pair fulfils $\mathcal{H}_n$.
\end{proof}

\begin{pro}\label{pro_un}
Let $\mu,\mu'\in\m$ with the same mass. Assume $\nu\in\m$ satisfies $\mu\leqcp \nu$ and $\mu'\leqcp\nu$. We have
\begin{align}\label{estim}
W(S^{\nu}(\mu),S^\nu(\mu'))\leq W(\mu,\mu')
\end{align}
\end{pro}

\begin{proof}
1. We first assume that $\mu$ and $\mu'$ are made of finitely many atoms of the same mass. We also assume $\down(\mu,\mu')\leqcp \nu$ and we denote this measure by $\tilde{\mu}$. As explain in Example \ref{topatom}, $\tilde{\mu}$ is a measure of the same type as $\mu$ and $\mu'$. Hence we can apply Lemma \ref{ahyo} to the pairs $(\mu,\tilde{\mu})$ and $(\mu',\tilde{\mu})$. Using Lemma \ref{stokanto} and Lemma \ref{topkanto} we can compute as follows
\begin{align*}
W(\mu,\mu')&=W(\mu,\tilde\mu)+W(\mu',\tilde\mu)\\
&=W(S^\nu(\mu),S^\nu(\tilde\mu))+W(S^\nu(\mu'),S^\nu(\tilde\mu))\\
&\geq W(S^\nu(\mu),S^\nu(\mu'))
\end{align*}

2. If $\tilde{\mu}\leqcp \nu$ does not hold, we have $\tilde{\mu}\leqcps \nu$ so that there exists $(\nu_n)_n$ as in Lemma \ref{add_queue}. We have $\tilde{\mu}\leqcp \nu+\nu_n$ where $\sup\spt\nu_n$ tends to $-\infty$ and the computation above leads to $W(S^{\nu+\nu_n}(\mu),S^{\nu+\nu_n}(\mu'))\leq W(\mu,\mu')$. Therefore with Lemma \ref{charge_a_droite}, we obtain \eqref{estim}.

3. For general $\mu,\mu'$ of the same mass $m$, that we assume to be 1, we approach them in $\p$ by measures $\mu_n\leqc\mu$ and $\mu'_n\leqc \mu'$ with the same barycenter, obtained as the sum of $2^n$ atoms of mass $m/2^n$. We do it in the following way: $\mu_n$ is defined as $\sum_{k=1}^{2^n}\frac{1}{2^n}\delta_{x_k}$ where 
$$x_k=2^n\int_{k/2^n}^{(k+1)/2^n} G_\mu(t)\dd \Lg(t).$$
The quantile function associated with $\mu_n$ is constant on each $]k/2^n,(k+1)/2^n]$ with value the mean of $G_\mu$ on this interval. We recognise for the filtration made of the dyadic intervals of $]0,m]$, the martingale associated with the random variable $G_\mu\in L^1(]0,1])$, the $L^1$-norm being the Kantorovich distance between the measures, as explained in Lemma \ref{stokanto}. Hence $(\mu_n)_n$ is non decreasing for the convex order and $\mu_n\longrightarrow \mu$ in $\m$. Thus applying Proposition \ref{limit_measure} we obtain the wanted estimate as $n$ goes to $\infty$.
\end{proof}
\begin{rem}
One can release the assumption to have atomic measures in Lemma \ref{ahyo} by using the approximation detailed in point 3 of the proof of Proposition \ref{pro_un}. Indeed, the stochastic order is stable in the weak topology so that $S^\nu(\mu)\leqs S^\nu(\mu')$ is true for general measures $\mu\leqs \mu'$.
\end{rem}

\subsubsection{Variations in $\nu$ and conclusion.}
\begin{pro}\label{pro_deux}
Let $\mu$ and $\nu,\nu'$ be elements of $\m$ such that $\mu\leqe \nu$, $\mu\leqe \nu'$ and $\nu(\R)=\nu'(\R)$. It holds
$$W(S^\nu(\mu),S^{\nu'}(\mu))\leq 2 W(\nu,\nu').$$
\end{pro}
\begin{proof}
1. We make first the additional assumption $\nu\leqs \nu'$ and we will prove $W(S^\nu(\mu),S^{\nu'}(\mu))\leq 2 W(\nu,\nu')$ in this case. Because of Proposition \ref{limit_measure}, we can assume without loss of generality that $\mu$ is of type $\sum_{i=1}^n\alpha_i\delta_{x_i}$ by using the same method as at step 3. of Proposition \ref{pro_un}. We can describe $S^\nu(\mu)$ as it is done in Example \ref{more_atoms} and introduce for this purpose a sequence $J_1\subseteq\cdots\subseteq J_n$. We have $S^\nu(\mu)=(G_\nu)_\#\lambda_{J_n}$ and for any $k$, $S^\nu(\sum^k\alpha_i\delta_{x_i})=(G_\nu)_\#\lambda_{J_i}$. We introduce now $S'=(G_{\nu'})_\#\lambda_{J_n}$ and $\mu'=\sum^k \alpha_i\delta_{x'_i}$ where $x_i'$ is the barycenter of $(G_{\nu'})_\#\lambda_{J_i\setminus J_{i-1}}$. As $\nu\leqs \nu'$, we have $G_\nu\leq G_{\nu'}$ and $S^\nu(\mu)\leqs S'$. Of course $x_i\leq x'_i$ so that $\mu\leqs \mu'$. According to the converse statement in Example \ref{more_atoms} we also have $S'=S^{\nu'}(\mu')$. Therefore using Proposition \ref{pro_un} for $\mu,\mu'\leqcp\nu'$.
\begin{align*}
W(S^\nu(\mu),S^{\nu'}(\mu))&\leq W(S^\nu(\mu),S^{\nu'}(\mu'))+W(S^{\nu'}(\mu'),S^{\nu'}(\mu))\\
&\leq W(\nu,\nu')+W(\mu,\mu')\leq 2W(\nu,\nu').
\end{align*}
Indeed, due to $\mu\leqs\mu'$ and $S^{\nu}(\mu)\leqs S^{\nu'}(\mu')$ we have
$$W(\mu,\mu')=W(S^\nu(\mu),S^{\nu'}(\mu'))=\int_{J_n}(G_{\nu'}-G_\nu)\dd \lambda\leq\int_0^{\nu(\R)}(G_{\nu'}-G_\nu)\dd \lambda=W(\nu,\nu').$$

2. We assume now $\mu\leqe \Top(\nu,\nu')$. In this case we use the triangle inequality, point 1 and Lemma \ref{topkanto} so that we can establish
$$W(S^\nu(\mu),S^{\nu'}(\mu))\leq 2 W(\nu,\nu').$$

3. Let us assume that $\mu\leqe \Top(\nu,\nu')$ does not hold. According to Lemma \ref{add_queue}, there exists $\gamma$ such that $\mu\leqe \Top(\nu,\nu')+\gamma$. But as stated in Lemma \ref{Top_com}, $\Top(\nu,\nu')+\gamma$ is also $\Top(\nu+\gamma,\nu'+\gamma)$ and $\mu\leqe \nu+\gamma$ as well as $\mu\leqe \nu'+\gamma$. Therefore according to the previous paragraph
$$W(S^{\nu+\gamma}(\mu),S^{\nu'+\gamma}(\mu))\leq 2 W(\mu,\mu').$$
Note that in Lemma \ref{add_queue}, $\sup(\spt(\gamma))$, that is the upper bound on the support of $\gamma$ can be chosen arbitrary close to $- \infty$. Hence, letting $\sup(\spt \gamma)$ go to $-\infty$, Lemma \ref{charge_a_droite} permits us to conclude in the more general case.
\end{proof}

\begin{proof}[Proof of Theorem \ref{lips}]
We combine Proposition \ref{pro_un} and Proposition \ref{pro_deux} and simply use the triangle inequality
\begin{align*}
W(S^\nu(\mu),S^{\nu'}(\mu'))&\leq W(S^\nu(\mu),S^{\nu'}(\mu))+W(S^{\nu'}(\mu),S^{\nu'}(\mu')).\\
&\leq 2W(\nu,\nu')+W(\mu,\mu')
\end{align*}
This can only be written doing the further assumption $\mu\leqe\nu'$. 

In the general case, let $\gamma$ be such that $\mu\leqe\nu'+\gamma$. Of course we still have $\mu\leqe \nu+\gamma$ and $\mu'\leqe \nu'+\gamma$. The previous computation holds and writes
$$W(S^{\nu+\gamma}(\mu),S^{\nu'+\gamma}(\mu'))\leq 2W(\nu+\gamma,\nu'+\gamma)+W(\mu,\mu').$$
But $W(\nu+\gamma,\nu'+\gamma)=W(\nu,\nu')$. We conclude using the same method as at the end of the proof of Proposition \ref{pro_deux}. With the end of Lemma \ref{add_queue} we obtain a suitable sequence $(\gamma_n)_{n}$ and we use the end of Lemma \ref{charge_a_droite} for the convergence $W(S^{\nu+\gamma_n}(\mu),S^{\nu'+\gamma_n}(\mu'))\to W(S^{\nu}(\mu),S^{\nu'}(\mu'))$.
\end{proof}

The following example shows that the left-curtain $\curt$ is not Lipschitzian when considering the Kantorovich distances $W$ and $W^{\R^2}$. In other words a result like Corollary \ref{doubidou} does not hold.

\begin{ex} \label{Wass_plouf}
For any integer $n\geq 1$ and $\eps<1$, we consider four measures of mass $n$. The first two ones employ Dirac masses at some points $k\in\N$ and $\eps>0$.
\begin{align*}
\mu=\sum_{k=0}^{n-1}\delta_k\quad\text{and}\quad\mu'=\delta_{\eps}+\sum_{k=1}^{n-1}\delta_k.
\end{align*}
and the two others  are made of uniform measures
\begin{align*}
\nu=\Lg_{[-1/2,n-1/2]}\quad\text{and}\quad\nu'=\Lg_{[-1/2+\eps/n,n-1/2+\eps/n]}.
\end{align*} Note that $\mu\leqs\mu'$, $\nu\leqs\nu'$ and $W(\mu,\mu')=W(\nu,\nu')=\eps$. As the measures are pairwise in convex order we can define the curtain couplings $\pi=\curt(\mu,\nu)$ and $\pi'=\curt(\mu',\nu')$. We have
$$\pi=\sum_{k=0}^{n-1}\delta_k\otimes\Lg_{[k-1/2,k+1/2]}.$$
The expression of $\pi'_n$ is more intricate.
$$
\pi'_n=\delta_{\eps}\otimes\Lg_{[-1/2+\eps,1/2+\eps]}+\sum_{k=1}^{n-1}\delta_k\otimes\Lg_{A_{k,1}\cup A_{k,2}}
$$
where for any $k\leq n-1$, the set $A_{k,1}\cup A_{k,2}=[-1/2+\eps/(k+1),\,-1/2+\eps/k)]\cup[k-1/2+\eps/k,\,k+1/2+\eps/{k+1}]$ is the union of an interval of length $\frac\eps{k(k+1)}$ close to $-1/2$ and the interval of length $1- \frac\eps{k(k+1)}$ with barycenter close to $k$.

in $\R^2$ the set $\{k\}\times A_{k,1}$ is part of the support of $\pi'$. It has mass $\eps/{k(k+1)}$ and distance to $\spt\pi$ greater than $k/2$ (for the $\ell^1$ norm $\|(x,y)\|=|x|+|y|$. It is in fact close to $k$.). If follows 
$$W^{\R^2}(\pi,\pi')>\sum_{k=1}^{n-1}\frac{\eps}{2(k+1)}=\frac12\left(\sum_{k=2}^n\frac1{k}\right)\max(W(\mu,\mu'),W(\nu,\nu')).$$
Note that we can normalise in mass and space and get the same estimate for families of probability measures close to $\Lg_{[0,1]}$. After this normalisation, the sequences $\pi_n$ and $\pi'_n$ both converge to $(\id\otimes\id)_\#\Lg_{[0,1]}$ but the ratios $W^{\R^2}(\pi_n,\pi'_n)/W(\mu_n,\mu'_n)$ and $W^{\R^2}(\pi_n,\pi'_n)/W(\nu_n,\nu'_n)$ go to infinity faster than $\ln(n)/2$. This make it impossible to find an estimate like Corollary \ref{doubidou} for the Kantorovich distance in place of $Z$.
\end{ex}

\section{The peacock problem, examples and counterexamples}
\subsection{Definitions}
Following Hirsch, Profeta, Roynette, and Yor \cite{HPRY} we call peacock a path $(\mu_t)_t$ of measures of $\p$ that is non-decreasing in the convex order. The origin of the name is the translation from the French ``Processus Croissant pour l'Ordre Convexe", in short PCOC, that is pronounced ``peacock'' in English. More precisely in this chapter we are concerned with peacocks indexed on $[0,1]$ and we will assume the paths are right-continuous in the topology $\mathcal{T}_1$. Note that due to the transformation $\mu\mapsto u_\mu$ and its relationship with the topology and the convex order, it is the same as assuming that the path is c\`adl\`ag (``continu \`a droite, limitu \`a gauche'', right-continuous with left limit). Note also that the set of discontinuous times is countable. 
Moreover, as explained in Remark \ref{rem_topo}, when restricted to $\{\mu_t\in\p:\,t\in[0,1]\}$, due to the uniform integrability of the measures in this set, the topology $\mathcal{T}_1$ is the same as the topology $\mathcal{T}_{\mathrm{cb}}$. Still according to this remark, the same holds when one considers joint laws of measures in this set.

 The problems raised in \cite{HPRY} are the existence and construction of processes $(M_t)_{t\in[0,1]}$ that are martingales in the canonical filtration $(\mathcal{F}_t)_{t\in[0,1]}$ (generated by $M$) and with $\law(M_t)=\mu_t$ for any $t$. More precisely, in the original problem the 1-dimensional marginals should be given as the marginals of an auxiliary process $(X_t)_t$ that is not a martingale. In this case the filtration to be considered may be that generated by $X$. 
In keeping with the practice of other authors in related papers like \cite{HiRo12, HTT}, we prefer to start with $(\mu_t)_{t\in[0,1]}$ as initial data. The answer to the question of existence predated the birth of peacocks by almost 40 years. In \cite{Ke72}, Kellerer established that a martingale may be associated with any peacock and that this martingale may be assumed to be Markovian. This aspect is important in the work by Kellerer. Indeed he stresses the fact that the Markov composition of two joint laws along a common marginal is not a continuous operator. Set $f_k(x)=|x|+1/k$ defined on $[-1,1]$. The Markov process
$$(\omega,t)\in\{0,1\}\times[-1,1]\mapsto (1-\omega)f_k(t)$$
converges to a non-Markovian process as $k$ tends to infinity. In fact, in the limit process, the future from position $0$ at time $0$ depends on the past (see also the counterexample to Satz 14 in \cite{Ke72} where the marginals are constant). Kellerer stresses the fact that the replicated Strassen theorem may associate a Markovian martingale with a peacock indexed on $\N$ but because of the lack of continuity that we already mentioned, the result can not be easily carried over to $\R^+$ (or $[0,1]$). In this chapter we insist on the Markov property in Kellerer's theorem. We construct processes using the left-curtain coupling and the Markov composition. We address the questions of existence, uniqueness and Markovianity. 

Let us cite the very recent contribution by Henry-Labord\`ere, Tan, and Touzi \cite{HTT} that is closely related to ours. These authors establish the uniqueness of the limit ``curtain'' process provided by the left-curtain coupling. This is done under some conditions that permit them in particular to include the peacock associated with a continuous peacock made of Gaussian measures (the resulting martingale is not the Brownian motion but a pure downward jump local Levy model). Under this set of assumptions, they identify the limit process with the minimiser of a functional on the set of processes. Another important result with respect to uniqueness is that of M. Pierre \cite[Theorem 6.1]{HPRY}, presented in the context of peacocks by Hirsch and Roynette \cite{HiRo12}. It does not rely on the left-curtain coupling but, for a class of peacocks, on the uniqueness of a martingale satisfying a type of Fokker-Planck equation, where the $1$-marginals are prescribed (the process is even Markovian). Another important contributor is Lowther \cite{Low, Low1, Low2}. In his papers, the author introduces the so-called {\it Lipschitz-Markov property}, close to the original conception of Kellerer \cite[Definition 2 and 3]{Ke72}. His work is revisited in \cite{HiRoYo}. 

In this chapter we conduct a study of some representative examples of peacocks, both continuous (Proposition \ref{pro_uniform}) and discrete (Theorem \ref{Poisson_fini}). For some of them, we prove the existence and uniqueness of a limit ``curtain'' process in the Skorokhod topology, for arbitrary sequences of partitions. For the others (Paragraphs \ref{threepoints} and \ref{nopoints}) we prove that there exist non-Markovian limits. We conclude by suggesting some open questions.

\subsubsection{Notations}
We consider a peacock $[\mu]:=(\mu_t)_{t\in[0,1]}$ and we build a process using a discretisation. More precisely, for any interval partition $\sigma=\{t_0,\ldots,t_Q\}$ of $[0,1]$ with $0=t_0<\cdots<t_Q=1$, there is a unique measure on the Skorokhod space $D([0,1])$ denoted by $P^{\mu,\sigma}$ such that the canonical process $X^{\mu,\sigma}_t:x\in(D,P^{\mu,\sigma})\mapsto x(t)\in\R$ is (almost surely) constant on every $[t_{k},t_{k+1}[$ and for every $k< Q$ the transport plan $(X_{t_k}\otimes X_{t_{k+1}})_\#P^{\mu,\sigma}$ is the left-curtain coupling between $\mu_{t_k}$ and $\mu_{t_{k+1}}$. In order to completely define $P^{\mu,\sigma}$ we need to specify that the canonical process $X^{\mu,\sigma}$ is a Markov process. The last important point is that $(X_t^{\mu,\sigma})_t$ is a martingale in the canonical filtration. Indeed, for $s\leq t$, $\E(X_t|X_s)=X_t$ because the transitions are martingale transport plans and $\E(X_t|\mathcal{F}_s)=\E(X_t|X_s)$ because the process is Markovian.

For a sequence of interval partitions $(\sigma^{(p)})_{p\in \N}$ with the mesh $|\sigma^{(p)}|$ going to $0$ as $p$ tends to $+\infty$, the sequence $P^{\mu,p}:=P^{\mu,\sigma^{(p)}}$ may converge in the weak topology of measures on the space $D([0,1])$ equipped with the Skorokhod topology. We denote by $\LIM([\mu])$ the set of possible limit measures on the Skorokhod space and call them \emph{limit curtain processes}. Actually any element $P\in\LIM([\mu])$ is relevant for the peacock problem because the canonical process $(X_t)_{t\in[0,1]}$ --- (i) satisfies $\law(X_t)=\mu_t$ for every $t$, (ii) is a martingale (see Proposition \ref{relevant} for the two points). One of our goals is to provide sufficient conditions for the set of possible limit curtain processes to be reduced to one element or to be not empty. The other goal is to examine whether there are Markovian or non-Markovian processes among the limit processes.

\subsubsection{Transformations}\label{transfo}

Given the peacock $[\mu]$ and a sequence $(\sigma^{(p)})_p$, if $(P^{\mu,p})_p$ converges as $p$ goes to infinity, the same holds for the peacock $[\phi_\#\mu]$ and $(P^{\phi_\#\mu,\sigma^{(p)}})_p$ where $\phi$ is the transformation $\phi:x\mapsto ax+b$ with $a>0$. The relation between the limits can be read in the canonical processes that are $(X_t)_t$ and $(aX_t+b)_t$.

Let us now consider another natural transformation $t\mapsto\mu_{\tau(t)}$ where $\tau:\R\to \R$ is a continuous non-decreasing function from $[0,1]$ to $[0,1]$. Using the uniform continuity of $\tau$ we see that $\LIM([\mu_\tau])$ is exactly the set of measures with canonical process $X_{\tau(t)}$ where $(X_t)_t$ is the canonical process of some $P\in\LIM([\mu_t])$.


\subsubsection{Finite dimensional topology}

In a similar way as $\LIM([\mu])$ we introduce $\LFD([\mu])$. For a right-continuous peacock $[\mu]$, it is the set of measures on $\R^{[0,1]}$ that are obtained for $(\sigma^{(p)})_{p\in\N}$ a sequence of partitions with mesh $|\sigma^{(p)}|$ going to $0$, as the limit  of $(P^{\mu,p})_p$ in the finite dimensional convergence.

Note with the following example that a limit curtain process for this topology may not satisfy the condition $\law(X_t)=\mu_t$. It may also be a measure on $\R^{[0,1]}$ that is not concentrated on $D([0,1])$. 

\begin{ex}\label{tiers}
Both assumptions can be shown with the peacock 
\begin{align*}
\mu_t=
\left\{
\begin{aligned}
&\delta_0&\text{ if }t<1/3,\\
 &(\delta_{-1}+\delta_1)/2&\text{ if }t\geq 1/3
 \end{aligned}
 \right.
\end{align*}
and $\sigma^{(p)}=\{k/2^p:\,k=0,\ldots,2^p\}$. One can see that for the corresponding limit process of $\LFD([\mu])$ satisfies $\law(X_{1/3})=\delta_0$, which make it impossible to be represented by a measure on c\`adl\`ag pathes.
\end{ex}

Nevertheless an element of $\LFD([\mu])$ is always a martingale as can be seen by adapting the proof Lemma 4.5 in \cite{HiRo12} by Hirsch and Roynette. In the next proposition, we adapt another proof by the same authors \cite[Theorem 3.2]{HiRo13} that relies on a Cantor diagonal argument and the classical regularisation of martingales (I first learnt this method from Beiglb\"ock \cite{Be_private}, see also \cite{Do}).
\begin{pro}\label{reprise}
Let $[\mu]=(\mu_t)_{t\in[0,1]}$ be a right continuous peacock. The set $\LFD([\mu])$ is not empty and contains at least one c\`adl\`ag process that is relevant for the peacock problem.

 More precisely, for any increasing (in the sense of inclusion) sequence of partitions $(\sigma^{(p)})_{p\in \N}$ such that $\cup_p\sigma^{(p)}$ contains the points of discontinuity of $t\mapsto \mu_t$, there exists a measure $P\in\LFD([\mu])$ and a subsequence of $(P^{\mu,p})$ converging to $P$ such that the canonical process associated with $P$ is a (maybe non-Markovian) c\'adl\'ag martingale with 1-dimensional marginals $\mu_t$ at any time $t\in[0,1]$.  

\end{pro}
\begin{proof}[Sketch of proof] We mostly follow the proof in \cite[Theorem 3.2]{HiRo13} that is done for dyadic nets. The only important difference concerning the partitions is that we include the times of discontinuity to our partition. This is important regarding to the finite dimensional convergence that we prove in the end of the proof after Lemma \ref{lemaux} but it is not for Theorem 3.2 in \cite{HiRo13} that is a pure existence theorem. At step (2) we adapt the proof as follows. We state $X^{(p)}_t=X^{(p)}_{\lfloor t\rfloor^p}$ instead of $0$ where $\lfloor t\rfloor^p$ is the greatest element of $\sigma^{(p)}$ that is smaller than $t$. Also we take the left-curtain transition with Markov transitions as explained in the beginning of this part instead of an arbitrary martingale provided by Strassen's Theorem. With our notations after step (2) we are really considering $P^{\mu,p}$ and the process $X^{\mu,p}$ where $\sigma^{(p)}$ may be non-dyadic and includes some discontinuities of $t\mapsto \mu_t$. We can follow the steps of Hirsch an Roynette until the end. We obtain a a kind of limit $P$ to a subsequence $(P^{\mu,\varphi(p)})_p$ of $(P^{\mu,p})_p$. The convergence happens as described now: when finitely many times of $\cup_p\sigma^{(p)}$ are selected the joint projection of $P^{\mu,\varphi(p)}$ on these marginals should converge to the corresponding projection of $P$. There is a unique c\`adl\`ag process that can be associated with $P$ seen as a measure on countably many copies of $\R$ (see step (5)). It is a martingale and the $1$-marginals are $\mu_t$ also at times $t\notin\cup_p\sigma^{(p)}$. 

The only element that must be added to Hirsch and Roynette's proof is the finite dimensional convergence for general times. For this purpose we need a lemma.
\begin{lem}\label{lemaux}
Let $(\nu_t)_{t\in[0,1]}$ be a peacock continuous at time $r$. For every $\eps>0$, there exists $\alpha$ such if $\max(|s-r|,|t-r|)<\alpha$, the sets $\M(\nu_{s},\nu_t)$ and $\M(\nu_t,\nu_{s})$ are contained in the ball of centre $(\id\otimes\id)_{\#}\nu_{r}$ and radius $\eps$ (for a distance metrifising $\Pi$).
\end{lem}
\begin{proof}
The result is due to the fact that $\M(\nu_{r},\nu_{r})=\{(\id\otimes\id)_\#\nu_{r}\}$. If $(\pi_n)_n$ is a sequence of martingale transport plans with both marginals converging to $\nu_r$, due to Prokhorov theorem it will converge to $(\id\otimes\id)_\#\nu_{r}$. The lemma follows from this remark.\end{proof}

Take $r_1,\ldots,r_j$ real times that are points of continuity of $t\mapsto\mu_t$, and for every $i\in\N$, $q_1^i,\ldots,q_j^i$ elements of $\cup_p\sigma^{(p)}$ such that for every $k\leq j$, $(q_k^i)_i$ is a non-increasing sequence converging to $r_k$. The sequence $u_{i,p}=\law(X^{\mu,\varphi(p)}_{q^i_1},\ldots,X^{\mu,\varphi(p)}_{q^i_j})$ depends on $i$ and on $p$. If $i$ tends to infinity for a fixed $p$, because paths of $D([0,1])$ are right continuous, the sequence $u_{i,p}$ converges to $\law(X^{\mu,\varphi(p)}_{r_1},\ldots,X^{\mu,\varphi(p)}_{r_j})$. If $p$ tends to infinity, for a fixed $i$, it converges to $\law(X_{q^i_1},\ldots,X_{q^i_j})$ where $X$ is the canonical process associated with $P$.  Moreover $\law(X_{q^i_1},\ldots,X_{q^i_j})$ tends to $\law(X_{r_1},\ldots,X_{r_j})$ when $i$ goes to infinity. We need to state the commutativity of limits $\lim_p(\lim_i u_{i,p})=\lim_i(\lim_p u_{i,p})$. Lemma \ref{lemaux} will provide the required uniformity. We consider it with the Prokhorov distance \cite[Section 6]{Bi} on the space of probability measures on $(\R^j,\|.\|_\infty)$ that we denote denote it by $d$. Take $\eps>0$. We claim that for sufficiently large numbers $p$ (greater than some $p_\eps$) the inequality 
$$d(\law(X^{\mu,\varphi(p)}_{r_1},\ldots,X^{\mu,\varphi(p)}_{r_j}),\law(X^{\mu,\varphi(p)}_{q^i_1},\ldots,X^{\mu,\varphi(p)}_{q^i_j}))<4\eps.$$
holds as soon as $i$ is sufficiently large (greater than some $i$ that does not depend on $p$). Indeed, for every $k=1,\ldots,j$, $\law(X^{\mu,\varphi(p)}_{r_k}, X^{\mu,\varphi(p)}_{q^i_k})$ is an element of $\M(\mu_{s},\mu_{t})$ where $s$ and $t$ have distance to $r_k$ smaller than $\max(|\sigma^{\varphi(p)}|,|q^i_k-r_k|)$. With the same notation $\alpha$ as in Lemma \ref{lemaux} we choose $p_{\eps,k}$ sufficiently large so that $|\sigma^{\varphi(p)}|\leq \alpha$ if $p>p_\eps$. If now we also have $|q^i_k-r_k|\leq \alpha$, the Prokhorov distance associated with the norm $\|.\|_{\infty}$ of $\R^2$ between $(X_{r_k}\otimes X_{q^i_k})_\#P^{\mu,\varphi(p)}$ and $(\id\otimes\id)_\#\mu_{r_k}$ is smaller than $\eps$. Hence considering a coupling as ensured by the Strassen-Dudley theorem \cite[Theorem 6.9]{Bi} one can couple these measures in a close way. Therefore $P^{\mu,\varphi(p)}(|X_{r_k}-X_{q^i_k}|<2\eps)\geq  1-\eps$, which means that for $p>p_\eps:=\max_k{p_{\eps,k}}$ the event $\{\max\limits_k|X_{r_k}-X_{q^i_k}|<2\eps\}$ holds with $P^{\mu,\varphi(p)}$-probability greater that $1-\eps$. We conclude using the triangle inequality that for such a value of $p$
\begin{align*}
&d(\law(X^{\mu,\varphi(p)}_{r_1},\ldots,X^{\mu,\varphi(p)}_{r_j}),\law(X_{r_1},\ldots,X_{r_j}))\\
\leq&d(\law(X^{\mu,\varphi(p)}_{r_1},\ldots,X^{\mu,\varphi(p)}_{r_j}),\law(X^{\mu,\varphi(p)}_{q^i_1},\ldots,X^{\mu,\varphi(p)}_{q^i_j}))\\
&+d(\law(X^{\mu,\varphi(p)}_{q^i_1},\ldots,X^{\mu,\varphi(p)}_{q^i_j}),\law(X_{q^i_1},\ldots,X_{q^i_j}))\\
&+d(\law(X_{q^i_1},\ldots,X_{q^i_j}),\law(X_{r_1},\ldots,X_{r_j}))\\
\leq&2\eps+\eps+\eps
\end{align*}
holds as soon as $i$ is sufficiently large. This proves the finite dimensional convergence.
\end{proof}

Similar arguments as the ones in Proposition \ref{reprise} permit us to prove that the elements of $\LIM([\mu])$ are all relevant.
\begin{pro}\label{relevant}
Let $P^{\mu,p}$ converge to $P$ in the Skorokhod topology. The canonical process $(X_t)_{t\in[0,1]}$ satisfies $\law(X_t)=\mu_t$ for every $t\in[0,1]$ and it is a martingale.
\end{pro}
\begin{proof}
With the convergence in the Skorokhod space, the finite dimensional convergence is also true for finitely many times selected in a set $E=[0,1]\setminus D$ where $D$ is countable (and $1\in E$). It follows that $(X_t)_{t\in E}$ is a martingale by using the same argument as before adapted from \cite[Lemma 4.5]{HiRo12}. But $(X_t)_{t\in[0,1]}$ and $(\mu_t)_{t\in[0,1]}$ are both right-continuous. It follows that $\law(X_t)=\mu_t$ is also satisfied for $t\in D$ and also that $(X_t)_{t\in[0,1]}$ is the regularisation on the right of the the martingale indexed on $E$ in the sense of the classical theory of continuous martingales (same argument as step (5) in \cite[proof of Theorem 3.2]{HiRo13}). It is a martingale.
\end{proof}

Finally notice that Example \ref{tiers} illustrates the fact that a sequence $P^{\mu,p}$ may converge to two non-compatible limits for the Skorokhod and the finite dimensional topology.

In the following paragraphs we give examples of what can be $\LIM(\mu)$ for selected classes of peacocks.

\subsection{Uniform measures on intervals}\label{intervals}

\begin{figure}[ht]
\begin{center}
\includegraphics[width=6cm]{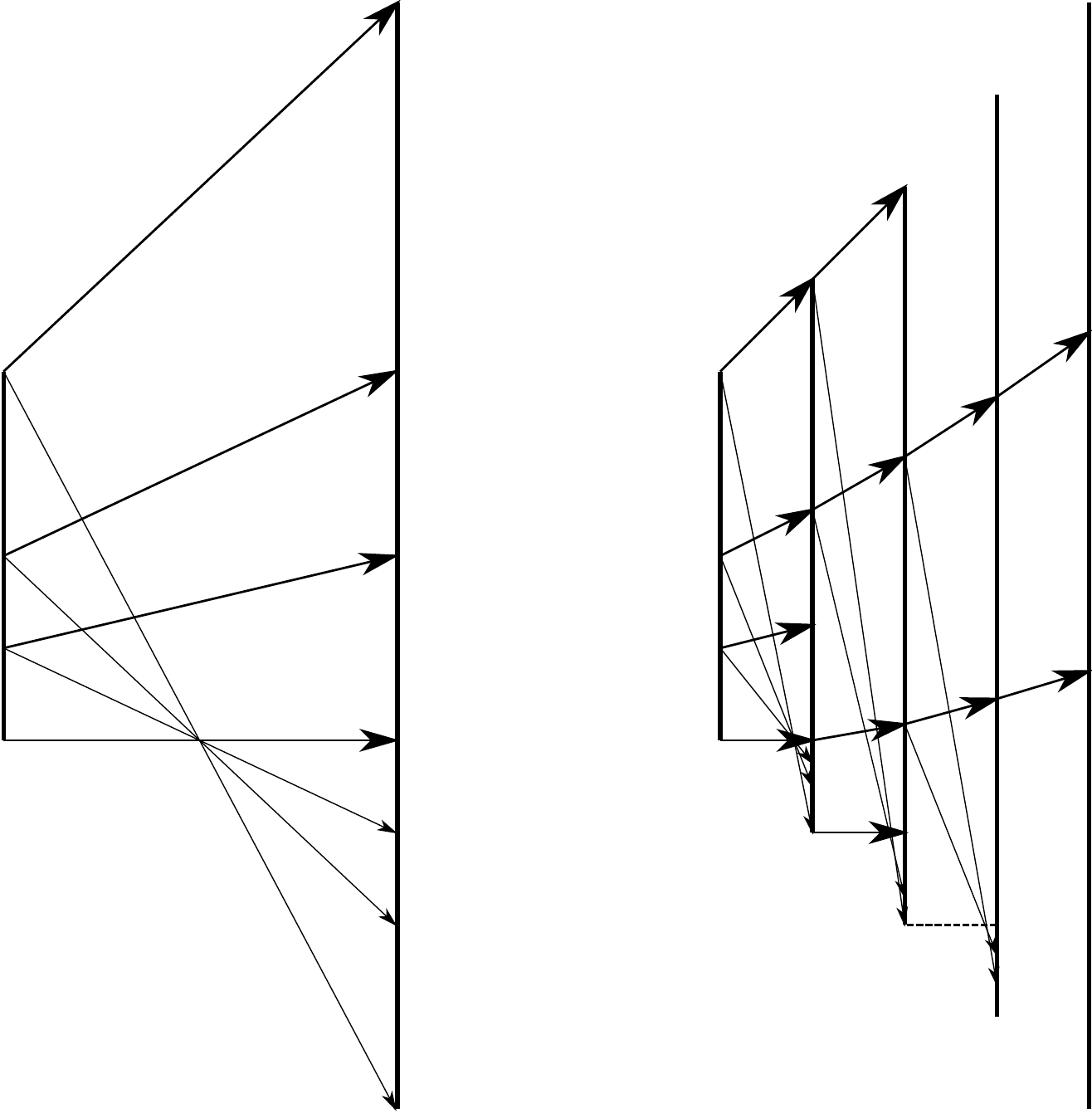}
\caption{Composition of curtain couplings for uniform measures.}\label{uniform_fig}
\end{center}
\end{figure}

The left-curtain coupling  $\pi_\lc$ from one uniform measure on a segment $\mu$ to another $\nu$ can be easily deduced from Definition \ref{lc}. Indeed, the shadows of $\mu|_{]-\infty,x]}$ in $\nu$ are also uniform measures. We give a description of the resulting coupling. As explained in paragraph \ref{transfo} it is invariant by translation and scaling so that we only need to explain it for $\mu=\Lg_{[0,1]}$ and $\nu=\Lg_{[-a,1+a]}$. The left-curtain coupling of those measures can be described with two linear maps. The submeasure $\frac{1+a}{1+2a}\Lg_{[0,1]}\leqp\mu$ is mapped linearly on $[0,1+a]$ and the remainder $\frac{a}{1+2a}$ linearly on $[0,-a]$. The coupling can also be described with random variables. Let $X$ be uniform on $[0,1]$ and $Z$ be an independent Bernoulli variable $Z\hookrightarrow \B(a/(1+2a))$. Then define $Y=(1+a)X$ if $Z=0$ and $Y=-a X$ if $Z=1$. Then $\pi_\lc=\law(X,Y)$. 

We note that the subsequent Proposition \ref{pro_uniform} can also be seen as a
consequence of the theorems in \cite{HTT} where the authors show that under
certain assumptions a continuous peacock gives rise to a unique limit
process that is a pure downward jump local Levy model. However, we
think that our example is worth presenting, because it appears
particularly canonical and is not considered in \cite{HTT}.

\begin{pro}\label{pro_uniform}
For every $t\in[0,1]$, let $\mu_t$ be uniform on $[-\exp(2t)/2,\exp(2t)/2]$. There is a unique curtain limit process to the peacock $[\mu]=(\mu_t)_{t\in[0,1]}$ and it can be described as follows: choose an initial point $X_0$ uniformly on $[-1/2,1/2]$ and independently a one-dimensional Poisson point process of intensity $1$ with times $0=T_0<T_1<\ldots<T_N<1$ where $N\hookrightarrow\mathcal{P}(1)$. The random path is defined as follow
\begin{align*}
X(t)=
\begin{cases}
\exp(2t)/2-(1/2-X_0)\exp(t)&\text{if }t\in[0,T_1[\\
\exp(2t)/2-\exp(t+T_i)&\text{if }t\in[T_i,T_{i+1}[\\
\exp(2t)/2-\exp(t+T_N)&\text{if }t\in[T_N,1]\\
\end{cases}
\end{align*}
\end{pro}

Roughly describing the limit process, after a jump at time $T$, the trajectory starts a new increasing piece from position $-\exp(2T)/2$. Figure \ref{uniform_fig} is an illustration of the discrete process for a linear, instead of exponential, evolution of the length of the segments.  One reason for the space-time normalisation chosen in the proposition for the continuous peacocks made of uniform measures (they are all equivalent according to paragraph \ref{transfo}) is that up to a scaling factor, for any $h>0$ the transition kernel between $\mu_t$ and $\mu_{t+h}$ is independent from $t$. Moreover, for the same times the probability to jump down is $2^{-1}(1-\exp(2h))$, which is equivalent to $h$ as $h$ tends to $0$. 

Our proof relies on the Euler approximation method and the approximation of the classical Poisson point  process by Bernoulli processes.

\subsubsection{Euler approximation method}\label{EulNota}

We consider a continuous transition function $T:(s,t,x)\in [0,1]\times [0,1]\times\R \mapsto \R$ defined for $s\leq t$ such that the limit
$$V(t,x)=\lim_{h\to 0^+}T(t,t+h,x)/h$$
exists. We denote by  $R_T$ the rest of the Taylor expansion
$$R_T(s,t,x)=T(s,t,x)-(t-s)V(s,x).$$
In the next proposition we will compare for a given partition $\sigma:0=t_0<t_1<\cdots<t_N<1=t_{N+1}$ and two initial points $x_0$ and $\bar{x}_0$, the solution $x(t)$ of the ODE
\begin{align}\label{Picard}
\left\{
\begin{aligned}
x(0)&=x(t_0)=x_0\\
\dot{x}(t)&=V(t,x(t))
\end{aligned}
\right.
\end{align}
to the Euler scheme starting in $\bar{x}_0$:

\begin{align*}
\left\{
\begin{aligned}
\bar{x}_0&=x_0\\
\bar{x}_{k+1}&=\bar{x}_k+T(t_k,t_{k+1},\bar{x}_k)
\end{aligned}
\right.
\end{align*}

The comparison can be done at discrete times $t_k$ between $\bar{x}_k$ and $x(t_k)$ but also in continuous time associating the c\`adl\`ag function $\bar{x}$ defined by $\bar{x}(t)=\bar{x}_k$ on $[t_k,t_{k+1}[$  with $(\bar{x}_k)_{k=0,\ldots,T}$. The proof follows the classical line for the convergence of the Euler scheme in numerical analysis.

\begin{pro}\label{Euler}
Let $T$, $V$ and $R_T$ be the functions introduced above and assume that $V$ is continuous, bounded and that there exists $L>0$ such that
$$|V(y,t)-V(x,t)|\leq L|y-x|.$$

Let $R_V$ be the local truncature error in the approximation of the flow at first order. We assume the uniform estimates
$|R_T(s,t,x)|:=|V(x,t+h)-V(x,t)|\leq M(t-s)^2/2$ and $|R_V(t,t+h,x)|\leq Mh^2/2$
for some $M>0$.

Then
$$\|x-\bar{x}\|_\infty\leq F(|x_0-\bar{x}_0|,|\sigma|)$$
for a non-increasing function $F$ with limit $0$ in $(0,0)$. 
\end{pro}
Note that the hypothesis on $V$ ensure that \eqref{Picard} is in the scope of Picard-Lindel\"of Theorem.
\begin{proof}
We consider the one-step operation starting from $\bar{x}_k$ and $x_k$ on the interval $[t_k,t_{k+1}[$:
$$\bar{x}_{k+1}=\bar{x}_k+(t_{k+1}-t_k)V(t_k,\bar{x}_k)+R_T(t_k,t_k+1,\bar{x}_k)$$
and
\begin{align*}
x(t_{k+1})=x(t_k)+\int_{t_k}^{t_{k+1}}V(s,x(s))\dd s=x_k+(t_{k+1}-t_k)V(t_k,x_k)+R_V(t_k,t_{k+1},x_k)
\end{align*}
We take the difference and obtain
\begin{align*}
|\bar{x}_{k+1}-x(t_{k+1})|
&\leq |\bar{x}_k-x(t_k)|+(t_{k+1}-t_k)|V(t_k,\bar{x}_k)-V(t_k,x_k)|+|R_V|+|R_T|\\
&\leq |\bar{x}_k-x(t_k)| (1+L(t_{k+1}-t_k))+M(t_{k+1}-t_k)^2
\end{align*}
Using the fact that $\prod_{k=1}^n(1+h_k)\leq \exp(h_1+\cdots+h_n)$ for positive real numbers $h_k$. It follows for $n\leq N$
\begin{align*}
|\bar{x}_n-x_n|&\leq |\bar{x}_0-x_0|\exp(L.t_n)+\sum_{k=1}^n M(t_k-t_{k-1})^2\exp(L(t_n-t_{k}))\\
&\leq |\bar{x}_0-x_0|\exp(L.t_n)+M|\sigma|\exp(L.t_n)\leq (|\bar{x}_0-x_0|+M|\sigma|)\exp(L.t_n)
\end{align*}
Hence $\|\bar{x}-x\|_\infty\leq (|\bar{x}_0-x_0|+M|\sigma|)\exp(L.t_n)+\|V\|_{\infty}.|\sigma|$
\end{proof}

\subsubsection{Poisson point process}

We state the following result without proof. It states that it is possible to couple a Bernoulli process and a Poisson process. We invite the reader to consult \cite{BHJ} on the Poisson approximation.

\begin{lem}\label{ppp}
Let $(\sigma^{(p)})_{p\in \N}$ be a sequence of interval partitions with the mesh $|\sigma^{(p)}|$ going to $0$ as $p$ tends to $+\infty$. There exists a probability space on which one can define an increasing sequence of random variables $(T_i)_{i\in\N^*}$ and for every $n\in\N$, an increasing sequence $(T_i^{(p)})_{i\in\N^*}$, such that
\begin{itemize}
\item $\{T_i:\,T_i<1\}$ realises a Poisson point process of intensity $1$ on $[0,1]$, say $T_0=0$, the $(T_{i}-T_{i-1})_{i\in\N^*}$ are independent and have the exponential distribution of parameter $1$.  

\item for every $p>0$, $\{T^{(p)}_i:\,T^{(p)}_i\leq 1\}$ is a Bernoulli process defined as follow. For the partition  $\sigma^{(p)}=\{t_0,\ldots,t_{Q_p}\}$ where $0= t_0<\cdots<t_{Q_p}=1$,  let $(B_k^{(n)})_k$ be a sequence of independent Bernoulli variables of parameter $2^{-1}(1-\exp(2(t_k-t_{k-1})))$. The time $T^{(p)}_i$ is $t_k$ where $k$ is the range of the $i$-th variable $B_k=1$. If such a range does not exist $T^{(p)}_i=+\infty$.

\item For every $\eps>0$, if $N=\#\{T_i:\,T_i<1\}$ and $N^{(p)}=\#\{T^{(p)}_i:\,T^{(p)}_i\leq 1\}$ we have
\begin{align*}
\P\left(\{N=N^{(p)}\}\text{ and }\{i\leq N\Rightarrow|T^{(p)}_i-T_i|\leq \eps\}\right)\longrightarrow_{n\to\infty} 1.
\end{align*}
\end{itemize}
\end{lem}

The next lemma concerns the trajectories of the limit process suggested in Proposition \ref {pro_uniform}. These are for $S\in [0,1]$,
$$g_{S}:t\in[S,1]\mapsto \frac12\exp(2t)-\exp(S+t)\in\R.$$

\begin{lem}\label{equiunif}
Consider $S<T$ and $S'<T'$, four times in $[0,1]^2$.
We have
$$\max_{u\in[0,1]}\left|g_{S}(S+u(T-S))-g_{S'}(S'+u(T'-S'))\right|\leq 10(|S'-S|+|T'-T|).$$
\end{lem}
\begin{proof}
One can consider the partial derivatives in $S$ and $T$ of $g_S(S+u(T-S))$ for a fixed $u\in[0,1]$. The norm of these derivatives is bounded by $\e^2$ and the relation holds if $S=S'$ and $T=T'$. It is enough for the Lipschitz bound. 
\end{proof}

\begin{proof}[Proof of Proposition \ref{pro_uniform}]
We prove that for a sequence $\sigma^{(p)}$, the measure $P^{\mu,p}$ converge in the Skorokhod topology to the law $P$ of the process described in the statement of the proposition. Our strategy is to use the Prokhorov distance associated to the Skorokhod distance. In other words for every $\eps>0$, we want to couple $P$ and $P^{\mu,p}$ in $D([0,1])\times D([0,1])$ using a coupling $\Theta$ such that with probability greater that $1-\eps$ the (Skorokhod) distance between $(X_t)_{t\in[0,1]}$ and $(X^{\mu,p}_t)_{t\in[0,1]}$ is smaller than $\eps$. 

It is also correct to perform the coupling in another probability space and this is what we will do with the probability space of Lemma \ref{ppp} together with a uniform random random value $X_0\hookrightarrow\mathcal{U}([-1/2,1/2])$ independent from this space. We construct the process $X$ as explained in the statement of the proposition, that is a random path starts in $X_0$ at time $0$ and jump at times $T_1,\ldots,T_N$. After the $i$-th jump, the trajectory is $g_{T_i}$.

Before we describe the piecewise constant process $X^{\mu,p}$, let us introduce an intermediate process $Y^p$. A random path starts at point $X_0$ and jumps at each time $t_k$ if and only if the interval $[t_{k-1},t_k[$ contains some $T_i$. After a jump the trajectory is $g_{t_k}$.

The process $X^{\mu,p}$ does not directly follows the trajectories $g_S$ but it is a discretisation of those trajectories in the sense of Proposition \ref{Euler}. A random path starts in $X_0$. It is constant on each interval $[t_{k-1},t_k[$. At time $t_k$ it jumps down into $[g_{t_{k-1}}(t_{k-1}),g_{t_{k}}(t_{k})]$ with probability $2^{-1}(1-\exp(2(t_k-t_{k-1})))$, which is small. In the other case it does a small jump up from $x=X^{\mu,p}_{t_{k-1}}$ to
$$T(t_{k-1},t_k,x)=(x+2^{-1}\exp(2t_{k-1}))\frac{1+\exp(2\delta t_k)}{2}-2^{-1}\exp(2t_{k-1})$$
where $\delta t_k=t_k-t_{k-1}$. The vector field $V$ corresponding to this transition $T$ with respect to the definitions of paragraph \ref{EulNota} is $V(t,x)=x+2^{-1}\exp(2t)$. Note that $V$ is $1$-Lipschitz in $x$ and continuous in $t$. The solutions of the ODE \eqref{Picard} are of the form $\exp(2t)/2-C\exp(t)$ where $C$ is a constant. For $C=\exp(S)$ we recover $g_S$. The trajectories of the flow starting from $[-1/2,1/2]$ at time $0$ or from $-2^{-1}\exp(2S)$ for some $S\in[0,1]$ at time $S$ are bounded and $V$ is also bounded for $(t,x)$ in a bounded set.

We can now conclude explaining that with high probability the trajectories of $X$ are close to the ones of $Y$ and that the trajectories of $Y$ are close to the one of $X^{\mu,p}$. Of course this holds if $p$ is sufficiently large. For the first estimate we consider the event $\{N=N^{(p)}\}\cap\{i\leq N\Rightarrow|T^{(p)}_i-T_i|\leq \eps\}$ and define $\lambda:[0,1]\mapsto [0,1]$ as the piecewise linear and continuous change of time that fixes $\{0,1\}$ and maps each $T_i$ for $i\leq N$ on $T^{(p)}_i$. With this $\lambda$ used in the definition of the Skorokhod distance and Lemma \ref{equiunif} we see that the Prokhorov distance between $\law(X)$ and $\law(Y^p)$ is smaller than $10\eps$. As $\eps$ can be chosen arbitrary, this distance tends to $0$ as $p$ tends to infinity.

The distance between $\law(Y^p)$ and $\law(X^{\mu,p})$ also tends to zero because we can use Proposition \ref{Euler} in order to compare without time wiggling the piecewise constant trajectories to the pieces $g_{T_i^{(p)}}$ of $Y^p$. The beginning of the first trajectories lies in $[g_{t_{k-1}}(t_{k-1}),g_{t_{k}}(t_{k})]$ and the other starts in $g_{t_k}(t_k)$ so that the distance between these points tends to zero together with $|\sigma^{(p)}|$ and the precise expression of $F$ at the end of the proof of Proposition \ref{Euler} permits us to certify that the upper bound is uniform over all pieces of $Y^p$.
\end{proof}

\subsection{Finitely supported measures}\label{sec_finite}

Let $\V$ be the set of vectors 
$$(X;A)=(x_1,\ldots,x_n,a_1,\ldots a_n)\in\R^{2n}$$ such that $\sum a_i x_i=0$ and $\sum a_i=1$. Every vector of $\V$ can be associated with a signed measure $\sum_{i=1}^n a_i\delta_{x_i}$. Let $(x_1,\ldots,x_n,a_1,\ldots,a_n)$ and $(Y;B)=(y_1,\ldots,y_n,b_1,\ldots b_n)$ be two elements of $\V$. As a function of $(X;A)$ and $(Y;B)$ let now $\Gamma\subset M_{n\times n}(\R)$ be the subspace of matrices satisfying
\begin{align}\label{systeme}
\left\{
\begin{aligned}
&M1=A,\\
&1^\mathrm{T}M=B^\mathrm{T},\\
&MY=\diag(a)X.
\end{aligned}
\right.
\end{align}
where $1$ stays for the vector $(1,\dots,1)^\mathrm{T}$, $\diag(a)$ is the diagonal matrix with entries $a_1,\ldots,a_n$ and $A,\,B,\,Y$ and $X$ are columns.
\begin{lem}\label{Grasm}
With the notations above, assume that the entries of $X$ are all different and that the same holds for $Y$. 
The affine space $\Gamma\subset M_{n\times n}(\R)$ has dimension $(n-1)(n-2)$ and the map
\begin{align*}
f:((X;A),(Y;B))\in\V^2\mapsto \Gamma(X,A,Y,B)\subset M_{n\times n}(\R)
\end{align*}
is analytic. Here $f$ is a map into the affine Grassmanian of affine spaces of dimension $(n-1)(n-2)$ included in $\R^{n\times n}\equiv M_{n\times n}(\R)$ .
\end{lem}
\begin{proof}
We can prove that the application that maps $M\in \Gamma$ to the submatrix made of the $n-1$ upper rows and the $n-2$ left-more columns is an affine bijection with $M_{(n-1)\times (n-2)}$. Indeed there always exists a way to complete such a matrix to an element of $W$ and this way is unique. We first consider the $n-1$ upper rows together with the first and third constraint of \eqref{systeme}. On each line we obtain a $2\times2$ linear system to solve and the solution is unique because of $|\begin{smallmatrix}1&x_{n-1}\\1&x_n\end{smallmatrix}|\neq 0$. We complete the row in the unique possible way according to the second constraint and we have still two relations on the lower row that need to be checked. These relations rely on the definition of $\V$. First, we already have $\sum\sum m_{ij}=\sum b_j=1=\sum a_i$ and $\sum_j m_{ij}=a_i$ for every $i\leq n-1$. It follows $\sum m_{nj}=a_n$. Second, we have $\sum_j m_{ij} y_j=a_i x_i$ for every $i\leq n-1$ and we want to prove it for $i=n$. This follows by subtracting these $n-1$ relations to $\sum b_j y_j=\sum a_i x_i$.    
\end{proof}

In this section we are interested in defining a limit curtain coupling for a peacock 
$$\mu_t=\sum a_i(t)\delta_{x_i(t)}$$
where the entries of the vector $(x_1,\ldots,x_n,a_1,\ldots,a_n)(t)\in \V$ are real analytic functions of time and furthermore satisfy $a_1(t),\ldots,a_n(t)>0$ and $x_1(t)<\ldots<x_n(t)$ for every $t$.

The fact that the measures are in the convex order implies that for $s\leq t$ the subspace $\Gamma_{st}$ associated with $(X_s;A_s)$ and $(X_t;A_t)$ contains a matrix with non-negative entries. Indeed, the conditions defining $\Gamma_{st}$ are equivalent to the ones of $\M(\mu_s,\mu_t)$. More precisely the affine map
$$\sum_{1\leq i,j\leq n} m_{ij}\delta_{(x_i(s),x_j(t))}\in \M(\mu_s,\mu_t)\mapsto (m_{ij})_{1\leq i,j\leq n}\in\Gamma_{st}$$
is onto and has image $\Gamma_{st}\cap (\R^+)^{n\times n}$, the subset of non-negative matrices of $\Gamma_{st}$. Hence we can identify it with $\Pi_M(\mu_s,\mu_t)$.

According to Proposition \ref{synthese}, the matrix corresponding to the left-curtain coupling is the unique minimiser for fixed $\mu_s$ and $\mu_t$ of the transport cost function $C:M\mapsto\sum_{ij} m_{ij} (1+\tanh(-x_i))\sqrt{1+y_j^2}$ to be minimised on $\Gamma_{st}\cap (\R^+)^{n\times n}\equiv \Pi_M(\mu_s,\mu_t)$. We denote the corresponding matrix by $M(s,t)$.  Note that $(s,t)\mapsto M(s,t)$ is defined on $\mathcal{T}=\{(s,t)\in[0,1]^2:\,s\leq t\}$ and that it is continuous on $\mathcal{T}$, for instance because of Theorem \ref{corstab}. We also have $M(s,s)=\diag(a_1(s),\ldots,a_n(s))$. As $C$ is linear on $\Gamma_{st}$, the matrix $M(s,t)$ is an extreme point of $\Pi_M(\mu_s,\mu_t)\equiv \Gamma_{st}\cap (\R^+)^{n\times n}$. Hence we can deduce that $M(s,t)$ satisfies at least $(n-1)(n-2)$ relations of type $m_{i,j}=0$ that are independent from the ones defining $\Gamma_{st}$.

We also notice that the chain
\begin{align*}
\sum_{1\leq i,j\leq n} m_{ij}(s,t)\delta_{(x_i(s),x_j(t))}&\mapsto M(s,t)\\
&\mapsto \tilde{M}(s,t):= \diag(a)^{-1}(s)M(s,t)
\end{align*}
permits us to replace a martingale transport plan with a stochastic matrix with $n$ states because the sum of the entries on the $i$-th row is no longer $a_i(s)$ but $1$. Moreover $a(t)^\mathrm{T}=a(s)^\mathrm{T}\tilde{M}(s,t)$. Each state $i=1,\ldots, n$ represents a trajectory $t\mapsto x_i(t)$. 

Given an interval partition $\sigma^{(p)}=\{t_0,\ldots,t_{Q_p}\}$ of $[0,1]$ with $0=t_0<\cdots<t_{Q_p}=1$ of the interval $[0,1]$, we introduce the coherent family of $(\bar{A}^{(p)}(s,t))_{0\leq s\leq t\leq1}$ in the following way. If $s\in[t_i,t_{i+1}[$ and $t\in[t_j,t_{j+1}[$ the transition matrix between those times is
$\bar{A}^{(p)}(s,t)=\tilde{M}(t_i,t_{i+1})\cdots \tilde{M}(t_{j-1},t_{j})$. It sends the distribution of mass $(a)(t_i)$ to $(a)(t_{j})$. We will prove in Proposition \ref{Poisson_fini} that $\bar{A}^{(p)}(s,t)$ converges to a certain $A_{st}$ when $|\sigma^{(p)}|$ goes to zero. This will in particular prove that $A_{st}$ is a stochastic matrix that sends the distribution $a(s)$ to $a(t)$.

\begin{lem}\label{lemlocal}
There exists a constant $C$ such that for every $\xi^-,\,\xi^+$ and any finite increasing sequence $(\theta_k)_{k=0}^K$ in $[\xi^-,\xi^+]$, one has 
\begin{align*}
\left\|\tilde M(\theta_{0},\theta_{1})\cdots \tilde M(\theta_{K-1},\theta_{K})-\id_n\right\|<C(\xi^+-\xi^-)
\end{align*}
\end{lem}
\begin{proof}
Let us first prove the lemma in the case where the finite sequence $\theta_n$ is just $\theta_0=\xi^+$ and $\theta_1=\xi^-$. The product of transition matrices is $\tilde{M}(\xi^-,\xi^+)$ that we simply note $\tilde{M}$. Therefore due to the shape of the left-curtain couplings we can claim $\tilde{M}_{ij}=0$ for $j>i+1$ if $h:=(\xi^+-\xi^-)$ is sufficiently small. Indeed the shadow of $\sum_{l=1}^i a_l(\xi^-) \delta_{x_l(\xi^-)}$ must be close to $\sum_{l=1}^i a_l(\xi^+) \delta_{x_l(\xi^+)}$ when $h$ is small. More precisely considering the centres of mass of these measures only a mass of $O(h)$ is sent to the atoms $a_{i+1}(\xi^+)\delta_{x_{i+1}(\xi^+)},\ldots,a_{n}(\xi^+)\delta_{x_n(\xi^+)}$ and this bound $O(h)$ can be chosen uniformly in $\xi^-$. Because of Proposition \ref{ombre_a_droite} if $h$ is sufficiently small this part of the shadow can only be in $x_{i+1}$. Hence the claim on $\tilde{M}_{ij}$ holds. With similar arguments and using what has already been proved we obtain that for $j<i$ the entry $\tilde{M}_{ij}$ is also $O(h)$ uniformly in $\xi^-$. Indeed the measure $a_i\delta_{x_i}$ is transported to a measure of barycenter $a_i(\xi^-)$, only $O(h)$ is transported to $a_{i+1}$ and no mass goes on upper atoms. Therefore we have proved that for a given peacock, there exists some constant $c>0$ such that
$$\left\|\tilde{M}(\xi^-,\xi^+)-\id_n\right\|\leq c(\xi^+-\xi^-).$$

In the general case where $(\theta_k)_k$ is not reduced to two times, using the submultiplicativity of the operator norm, the estimate $1+x\leq \exp(x)$ and a telescopic sum we obtain
\begin{align*}
&\|(\tilde{M}(\theta_0,\theta_1)\cdots \tilde{M}(\theta_{K-1},\theta_K)-\id_n\cdots\id_n\|\\&\leq \sum_{k=0}^{n-1} c(\theta_{k+1}-\theta_k)\exp(c(\theta_K-\theta_0))\\
&\leq c.\e^c(\xi^+-\xi^-)
\end{align*} 
\end{proof}

We now prepare the identification of the limit curtain peacock with a family of transitions defined by ordinary differential equations in the space of stochastic matrices. However we start with transport matrices instead of stochastic matrices.
\begin{lem}\label{grandO}
There exists a countable and closed set $E\subset [0,1]$ such that for every $t\in[0,1]\setminus E$ and $h\geq 0$, there exists a matrix $N(t)=(n_{ij}(t))_{ij}$ that satisfies
\begin{align}\label{huit}
M(t,t+h)=\diag(a)(t)+hN(t)+O(h^2).
\end{align}
The sums of its entries on the rows is $0$ and on the $j$-th column is $\dd a_j/\dd t$. It also satisfies  $n_{ij}(t)=0$ for at least $(n-1)(n-2)$ entries $(i,j)$ among which those with $j>i+1$. The entries on the diagonal are non-negative and the other entries are non-positive. Furthermore $\|N(t)\|$ is uniformly bounded on $[0,1]\setminus E$.

Moreover, $E$ has finitely many accumulation points so that between two such points the elements  $(\theta_k)_k$ of $E$ are isolated. The map $t\mapsto N(t)$ is analytic on every $]\theta_k,\theta_{k+1}[$ and for every segment $S\subset ]\theta_k,\theta_{k+1}[$, there exists $C>0$ with the uniform estimate
\begin{align}\label{neuf}
\|M(t,t+h)-\diag(a)(t)-hN(t)\|\leq Ch^2
\end{align}
for $t\in S$ and $h>0$.
\end{lem}

\begin{proof}
We introduce an index $k\geq 1$ such that every $k\leq \binom{n^2}{(n-1)(n-2)}$ is associated with a subset $I_k$ of $(n-1)(n-2)$ entries of the matrices of $M_{n\times n}(\R)$. Moreover for every $s,t\in[0,1]^2$ we only consider the subsets $I_k$ such that the vectorial space $\Delta_k$ of matrices with the entries zero on $I_k$ is in direct sum with the vectorial part of $\Gamma_{st}$. The spaces $\Gamma_{st}$ are parallel for different values of $s$ so that we can denote the set of theses indices by $\mathcal{I}(t)$. The question whether $\Gamma_{st}$ is in direct sum with $\Delta_k$ is just depending analytically on the values of the functions $(x_j)_j$ at time $t$. The index $k$ will be an element of $\mathcal{I}(t)$ if and only if a certain determinant does not vanish at time $t$. Hence either $k$ is not an element of $\mathcal{I}(t)$ for every $t$ or it is, except finitely many times on $[0,1]$. 

For $k\in\mathcal{I}(t)$ we can now introduce the analytic map $(s,t)\mapsto M_k(s,t)$ where $\{M_k(s,t)\}=\Gamma_{st}\cap\Delta_k$ is the single point at the intersection. 

Let $]\xi^-,\xi^+[$ be an interval such that $\mathcal{I}(t)$ is the same for every $t$. With respect to the statement $\xi^-,\xi^+$ are elements of $E$ such that $S\cap E$ will be finite for every segment $S\subset ]\xi^-, \xi^+[$. The fact that there are finitely many points $\xi$ is due to the analyticity. We simply denote $\mathcal{I}(t)$ by $\mathcal{I}$. Let $t_0$ be in the interval. For every $(s,t)$ in a neighbourhood of $(t_0,t_0)$, the matrix $M(s,t)$ is continuous and it equals at least one $M_k(s,t)$ for $k\in\mathcal{I}$. We have also $M(t,t)=\diag(a)(t)$ because $\Gamma_{tt}\cap(\R^+)^{n\times n}$ is reduced to one point. Moreover the maps $(s,t)\mapsto M_k(s,t)-M_l(s,t)$ are analytic and the locus where they vanish close to $(t_0,t_0)$ is accordingly well-known (see for instance \cite[Chapter 6]{KP02}). Hence we deduce that there exists a neighbourhood of $t_0$ such that for every $s$ in this neighbourhood, there exist $k(s)$ and $\eps(s)>0$ with $M(s,s+h)=M_{k(s)}(s,s+h)$ for every $h\in[0, \eps(s)]$. Moreover the neighbourhood can be restricted so that $k$ is constant both for $s<t_0$ and for $s>t_0$. Finally the function $\eps$ can be chosen to be continuous. Using the compactness of the segments $S\subset ]\xi^-,\xi^+[$ we see that there exists at most finitely many accident times $\theta_k$ on $S$. Between two such times there exists $k\in\mathcal{I}$ with $M(t,t+h)=M_k(t,t+h)$ if $h$ is sufficiently small. The bound may be chosen uniformly on every segment included in $]\theta_k,\theta_{k+1}[$. Hence we obtain \eqref{huit} and \eqref{neuf} for $N=\dd M_k(t,t+h)/\dd h|_{h=0^+}$.

The statements on $N$ now follow from the system \eqref{systeme}, equation \eqref{huit}, the definition of $M_k$, Lemma \ref{lemlocal} and the structure of the zeros of $\tilde{M}(t,t+h)$ for small $h$ stated in the proof of this lemma.
\end{proof}

\begin{them}\label{Poisson_fini}
For every finite peacock $[\mu]$ concentrated on $n$ injective and analytic curves $t\in[0,1]\mapsto x_i(t)$ with analytic weight $a_i(t)$ (remind the setting after Lemma \ref{Grasm}), there is a unique limit curtain process and this process is Markovian. 

More precisely for $\tilde{N}_t=\diag(a_1(t),\ldots,a_n(t))^{-1}N_t$, the family $(A_{st})_{s\leq t}$ associated with the differential equations 
\begin{align}\label{semig}
\left\{
\begin{aligned}
\frac{\dd A_{su}}{\dd t}&=A_{su}\tilde{N}(u)\\
A_{ss}&=\id
\end{aligned}
\right.
\end{align}
defines a set of coherent transition matrices on a space of $n$ states. Together with the initial measure $(a_1,\ldots,a_n)(0)$ on this space, it defines a Markov process with c\`adl\`ag trajectories. This process is the limit in both the finite dimensional and the Skorokhod topology of any sequence $P^{\mu,p}$ associated with a sequence $(\sigma^{(p)})_p$ of partitions with mesh going to zero.

\end{them}

\begin{proof}
When written for the stochastic matrices, the equation \eqref{huit} becomes
$$\tilde M(t_0,t_0+h)=\id_n+h\tilde N(t_0)+O(h^2)$$
for $h\geq 0$ going to $0$ where $\tilde{N}(t)=\diag(a_1,\ldots,a_n)^{-1}(t)N(t)$. Lemma \ref{grandO} testifies that $\|\tilde{N}\|$ is uniformly bounded on $[0,1]\setminus E$ by some constant $c$. It follows that the system of differential equations \eqref{semig} is well defined, with $A_{st}=\id_n+\int_{s}^t A_{su} \tilde{N}_u \dd u$ and
$$\|A_{st}-\id\|\leq c(\e -1)(t-s)$$
because of the convexity of $\exp$ and $t-s\leq 1$. As $\tilde{N}(t)$ is uniformly bounded, it is the infinitesimal generator of a non-homogeneous Markov chain with states the curves $(x_i)_{i=1\ldots n}$. At time $t$, the rate for jumping from curve $x_i$ to curve $x_j$ is $\tilde{n}_{i,j}(t)$.

Recall that given an interval partition $\sigma^{(p)}=\{t_0,\ldots,t_{Q_p}\}$ of $[0,1]$ with $0=t_0<\cdots<t_{Q_p}=1$, we have introduced the coherent family of $(\bar{A}^{(p)}(s,t))_{0\leq s\leq t\leq1}$  before Lemma \ref{lemlocal}. Our first task is to prove that $\bar{A}^{(p)}(s,t)$ converges to $A_{st}$ when $|\sigma^{(p)}|$ tends to zero. This proves in particular that $A_{st}$ is a stochastic matrix that sends the mass row $(a)(s)$ to $(a)(t)$. Note that due to Lemma \ref{lemlocal} we also have $\|\bar{A}^{(p)}_{\xi,\xi+h}-\id\|\leq C(h+|\sigma^{(p)}|)$ for some constant $C$ only depending on the peacock.

 With respect to the notations introduced in paragraph \ref{EulNota} we can fix $s\in[0,1]$ and denote $\bar{A}^{(p)}(s,u)$ by $x(u)$. We obtain $T(u,u+h,x)=x(u)\tilde M(u,u+h)$ and $V(x,u)=x\tilde{N}(u)$. Proposition \ref{Euler} basically requires that $V$ is continuous in $u$ and Lipschitz continuous in $x$. The second condition is satisfied but the first one may not be true on every $[s,t]$. We introduce the set $E\subset [0,1]$ as in Lemma \ref{grandO}.
 
We first consider the case $[s,t]\subset ]\theta_k,\theta_{k+1}[$ with $(\theta_k)_k$ as in this lemma. Hence up to a time rescaling we can apply Proposition \ref{Euler}. We obtain that $\bar{x}(u)=\bar{A}_p(s,u)$ uniformly converges to $A_{su}$ for every $u\in [s,t]$ as $p$ goes to infinity. There is one difficulty to overcome that concerns the fact that the partition $\sigma^{(p)}$ may avoid the starting and end times $s,t$. This problem is fixed by the estimates of $\|A_{\xi,\xi+h}-\id\|$ above and $\|\bar{A}^{(p)}_{\xi,\xi+h}-\id\|$ (see Lemma \ref{lemlocal}) when $h$ is small. More precisely if $s'$ and $t'$ are respectively the greatest and smallest times in the partition $\sigma^{(p)}$ that satisfy $s'\leq s$ and $t\leq t'$, the matrices $A_{s't'}$ and $\bar A^{(p)}(s',t')$ tends to $A_{st}$ and $\bar A^{(p)}(s,t)$ respectively.

If now $E\cap [s,t]$ is not empty, due to the structure of $E$, it is possible to find finitely many $[s_k,t_k]$ that do not intersect $E$ such that the cumulated length $\sum (s_{k+1}-t_k)$ is arbitrarily small. Writing now $A_{st}=A_{s,s_1}A_{s_1t_1}A_{t_1s_2}\cdots A_{s_K t_K}A_{t_K t}$ and $\bar{A}^{(p)}(s,t)$ in a similar manner we obtain the estimate
\begin{align*}
&\|A_{st}-\bar{A}^{(p)}(s,t)\|\leq\\
& C\sum_{k=1}^K \|A_{s_k t_k}-\bar{A}^{(p)}(s_k,t_k)\|+C[(s_1-s)+(t_K-t)+\sum_{k=1}^{K-1}(s_{k+1}-t_k)].
\end{align*}
It follows from the fact that the first term tends to zero and the second can be chosen arbitrarily small that $\bar{A}^{(p)}(s,t)$ tends to  $A_{st}$ as $p$ goes to infinity.

For proving the convergence of $P^{\mu,n}$ in the finite dimensional topology {\it it is enough to prove the convergence for two times marginals}, which is what we have already done. This is a simple consequence of continuity of the product of finitely many real numbers.
 
In this sequel of this proof we explain how to prove the convergence in the Skorokhod topology. We prove below that for every $\eps>0$, if $n$ is sufficiently large, there exists $\Theta$ a measure on $D([0,1])\times D([0,1])$ with marginals $P^{\mu,p}$ and $P$ (the Poisson like process generated by $\tilde {N}$) such that with probability greater that $1-\eps$, the Skorokhod distance between the canonical marginal random elements of $D([0,1])^2$ is smaller than $\eps$ for the joint law $\Theta$. In other words we prove that the Prokhorov distance between $P^{\mu,p}$ and $P$ tends to zero. With a slight abuse we say and consider that a $P^{\mu,p}$-random trajectory $x$ is constant on $[s,t]$ if for some $k$ it starts close to $x_k$ at time $s$ and all the transitions are done from state $k$ to itself. In the case of a $P$-random trajectory, $x$ is constant on $[s,t]$ if it is continuous. In this case $x=x_k$ for some $k$. A more concrete way to justify this abuse is to introduce $\alpha>0$ and consider only partitions with a sufficiently small mesh so that $|x_k(t)-\bar{x}(t)|<\eps$ holds at any time for some $k$ and a unique $k$. In fact, due to the uniform continuity of the trajectories $(x_k)_{k=1,\ldots,K}$, the real $\eps$ may be chosen as small as one wants. Finally it is the same to prove that the Prokhorov distance tends to zero with $\R$ or $\{1,\ldots,n\}$ as state space. Consider a finite set $S\subset [0,1]$ and using the convergence in the finite dimensional topology, consider a sufficiently large $n$ and a coupling $\Theta$ such that with probability greater than $1-\eps/10$, we have $\|X_t-X_t^{\mu,p}\|\leq \eps$ for every time $t\in S$. Here $X$ and $X^{\mu,p}$ are the canonical processes and $\eps$ is sufficiently small to characterise the state in $\{1,\ldots,n\}$. Let us call jump the discontinuities of $X_t$ and the discontinuities in state of $X^{\mu,p}$. Lemma \ref{lemlocal} and the fact that $\tilde{N}_t$ is bounded permit us to claim that if the mesh of $S\cup\{0,1\}$ is sufficiently small, the probability that $X$ or $X^{\mu,p}$ has two or more jumps on some interval of the partition is smaller than $\eps/10$. We can also assume that this mesh is smaller than $\eps$, which is important for the horizontal distortion in the definition of the Skorokhod distance. On these conditions, we can easily prove that $\Theta$ is a convenient coupling for proving that the Prokhorov distance associated with the Skorokhod distance is smaller than $\eps$.

\end{proof}
 
\subsection{A discrete counterexample}\label{threepoints}
We show that not every element of $\LIM([\mu])$ is Markovian and that this set may have cardinal $\geq 2$.
Here we take the setting of the last paragraph with $n=3$ and $a_1=2a_2=2a_3=1/2$ but we do not assume $x_1<x_2<x_3$. In fact
\begin{align*}
x_1(t)&=9-t\\
x_2(t)&=8+2t\\
x_3(t)&=10
\end{align*}
so that $x_1<x_2<x_3$ on $[1/2,1[$ and $x_1<x_3<x_2$ on $]1,3/2]$. In this paragraph we prefer to parametrise the peacock on $[1/2,3/2]$ and we can do that because it has no theoretic importance as explained in paragraph \ref{transfo}. At time $1$, the points $x_2$ and $x_3$ meet in $10$. We will see that sequences of partitions not including the time $1$ all generate the same process independently of the sequence and that this is not a Markov process. On the contrary if the sequence includes the time $1$ (at least asymptotically) there exists a unique limit process independent of the sequence and this process is Markov. 

A computation permits us to state
\begin{align*}
M(t-h,t+h)&=
\left(
\begin{matrix}
1/2\frac{20-h}{20+h}&0&1/2\frac{2h}{20+h}\\
\frac{20h+5h^2}{(20+h)(40+6h)}&1/4\frac{2h}{20+3h}&1/4\frac{20-3h}{20+h}\\
1/4\frac{2h}{20+3h}&1/4\frac{20+h}{20+3h}&0
\end{matrix}
\right)\\
&=
\left(
\begin{matrix}
1/2&0&0\\
0&0&1/4\\
0&1/4&0
\end{matrix}
\right)+2h
\left(
\begin{matrix}
-1/40&0&1/40\\
1/80&1/80&-1/40\\
1/80&-1/80&0
\end{matrix}
\right)
+o(h),
\end{align*}
so that
$$\tilde{M}(t-h,t+h)=\left(\begin{matrix}
1/2&0&0\\
0&1/4&0\\
0&0&1/4
\end{matrix}
\right)^{-1} M(t-h,t+h)=\left(
\begin{matrix}
1&0&0\\
0&0&1\\
0&1&0
\end{matrix}
\right)+O(h)$$
Hence at the order zero a subdivision including the interval $[1-h,1+h[$ realises a permutation between the trajectories of $x_2$ and $x_3$. We admits without computation that we obtain the same permutation for any interval $[1-h,1+h'[$ with $h,h'>0$ tending to zero.

Let us now see what happens if we have two intervals $[1-h,1[$ and $[1,1+h'[$. At order zero the transport plan are
\begin{align*}
\left(
\begin{matrix}
1/2&0&0\\
0&1/4&1/4
\end{matrix}
\right)
\AND
\left(
\begin{matrix}
1/2&0\\
0&1/4\\
0&1/4
\end{matrix}
\right),
\end{align*}
which corresponds to transitions
\begin{align*}
\left(
\begin{matrix}
1&0&0\\
0&1/2&1/2
\end{matrix}
\right)
\AND
\left(
\begin{matrix}
1&0\\
0&1/2\\
0&1/2
\end{matrix}
\right),
\end{align*}
and after composition of the transitions we obtain
\begin{align*}
\left(
\begin{matrix}
1&0&0\\
0&1/2&1/2\\
0&1/2&1/2
\end{matrix}
\right)
\neq
\left(
\begin{matrix}
1&0&0\\
0&0&1\\
0&1&0
\end{matrix}
 \right).
\end{align*} 
Let us sum up. For partitions including time $1$ we observe the same Poisson-like process behaviour as in Theorem \ref{Poisson_fini}. On $[1/2,1[$ and $]1,3/2]$ the law of the limit process is exactly the same as the one obtained for the sequences of partitions avoiding time $1$. For the description of the limit process on $[1/2,3/2]$, we need one Bernoulli trial more at time $1$. With probability $1/2$ a locally continuous trajectory will follow $x_2$ when arriving in $1$ and with probability $1/2$ it will follow $x_3$. In the former case the probability for a random trajectory to locally equal $x_2$ or $x_3$ at time $1$ was zero.

\subsection{A continuous counterexample.}\label{nopoints}
The aim of this section is again to show that there may be elements of $\LIM([\mu])$ that are not Markovian processes. With respect to part \ref{threepoints}, the novelty is that all measures $\mu_t$ in this section are absolutely continuous, which means that this possibility is not due to atoms.

\begin{figure}[ht]
\begin{center}
\def\svgwidth{8cm}
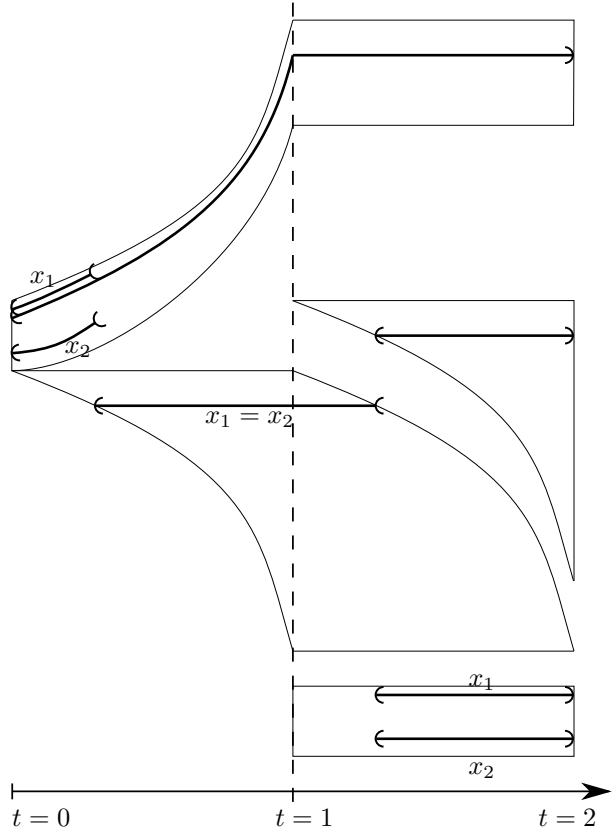
\caption{A non-Markovian limit process associated with a peacock with absolutely continuous $1$-marginals.}\label{counter}
\end{center}
\end{figure}

\subsubsection{First step: stocking}

Consider a peacock of the type $\mu_t=\mu^1_t+\mu^2_t$ where
\begin{itemize}
\item for every $t$, $\spt(\mu_t^1)\subset]-\infty,b]$ and it is the restriction of $\mu_1^1$ to $]a(t),b[$ with $a$ a decreasing function, 
\item for every $t$, $\mu_t^2$ is concentrated on $]b,+\infty[$.
\end{itemize}
We call such a peacock a \emph{stocking peacock}. We give later an example related to part \ref{intervals}. But let us first describe the shape of the left-curtain coupling $\pi_{\lc}$ between $\mu_s$ and $\mu_t$ for $s<t$. Due to $\mu^1_s\leqp\mu^1_t$, we have $S^{\mu^1_t}(\mu^1_s)=\mu^1_s$. It follows that $\pi_\lc$ is $(\id\otimes\id)_\#\mu^1_s+\pi$ where the marginals of $\pi$ are $\mu^2_s$ and $\mu^2_t+(\mu^1_t-\mu^1_s)$. We can conclude that a process $X^{\mu,\sigma}$ associated with a stocking peacock $[\mu]$ and a partition $\sigma$ is constant from the time it meets $]-\infty,b[$. 

\begin{ex}\label{stock}
Consider the limit process defined in paragraph \ref{intervals}. If we stop this martingale after the first jump, it is still a martingale. Considering then the time marginals, we obtain a peacock $\mu_t=\mu^1_t+\mu^2_t$ with 
$$\mu^1_t(\dd x)=\frac{1}{(-2x)^{3/2}}\Lg_{[-\e^{2t}/2,-1/2]}\dd x\quad\text{ and }\quad\mu^2_t=\e^{-2t}\Lg_{[\e^{2t}/2-\e^t,\e^{2t}/2]}.$$
This is clearly a stocking peacock for $a(t)=-\exp(2t)/2$ and $b=-1/2$. See the left part of Figure \ref{counter} for an illustration.
\end{ex}

We claim without details that this peacock has a unique curtain limit and that it is the stopped martingale itself. This fact is specific to this peacock. A proof can be derived from the techniques in paragraph \ref{intervals}.

\subsubsection{Second step: destocking}
Consider a peacock of the following type: $\mu_t=\mu_t^1+\mu_t^2+\mu_t^3$ where
\begin{itemize}
\item for every $t$, $\spt(\mu_t^1)\subset]-\infty,a[$ and $s\leq t\Rightarrow \mu^1_s\leqp\mu^1_t$,
\item $\mu_t^2$ is concentrated on $[a,b]$ and it is the restriction of $\mu_0^2$ to $[a,b(t)]$ with $b$ a decreasing function,
\item for every $t$, $\spt(\mu_t^3)\subset]b,+\infty[$ and it is the restriction of $\mu_1^3$ to $[c(t),+\infty[$ with $c$ a decreasing function.
\item for every $s,t$ if $s\leq t$ we have $\mu^2_s-\mu^2_t\leqc (\mu^1_t-\mu^1_s)+(\mu^3_t-\mu^3_s)$.
\end{itemize}
We call such a peacock a \emph{destocking peacock}. Note that it may also be a stocking peacock for the writing $\mu^1_t+(\mu^2_t+\mu^3_t)$. The name indicates that one is ``destocking'' the mass in $\mu^2_t$. Let us describe the transition between $\mu_s$ and $\mu_t$ given by the left-curtain coupling. For $x\leq b(t)$ we have $\mu_s|_{]-\infty,x]}\leqp\mu_t|_{]-\infty,x]}$ so that $S^{\mu_t}(\mu_s|_{]-\infty,x]})=\mu_s|_{]-\infty,x]}$. Note in particular $S^{\mu_t}(\mu_s|_{]-\infty,b(t)]})=\mu_s^1+\mu_t^2$ because $\mu_s|_{]-\infty,b(t)]}=\mu^1_s+\mu^2_t\leqp\mu_t$.  Recall $b(s)\geq b(t)$ and consider now the shadow of $\mu_s|_{]-\infty,b(s)]}=\mu_s^1+\mu_s^2$. Using Proposition \ref{shadowsum} and Lemma \ref{ombre_a_droite} we obtain
\begin{align*}
S^{\mu_t}(\mu_s^1+\mu^2_s)&=S^{\mu_t}(\mu_s^1+\mu^2_t)+S^{\mu_t-S^{\mu_t}(\mu_s^1+\mu^2_t)}(\mu^2_s-\mu^2_t)\\ 
&=(\mu^1_s+\mu^2_t)+S^{(\mu^1_t-\mu^1_s)+\mu^3_t}(\mu^2_s-\mu^2_t)\\
&=(\mu^1_s+\mu^2_t)+[(\mu^1_t-\mu^1_s)+(\mu^3_t-\mu^3_s)]\\
&=\mu_t-\mu^3_s.
\end{align*}
From these computation it follows that the left-curtain coupling is $(\id\otimes\id)_\#(\mu^1_s+\mu^2_t+\mu^3_s)+\pi$ where $\pi$ is a martingale coupling of marginals $\mu^2_s-\mu^2_t$ and $(\mu^1_t-\mu^1_s)+(\mu^3_t-\mu^3_s)$.

See the right part of Figure \ref{counter} for an illustration of a locally destocking peacock. The support of $\mu_t^3$ is the union of two intervals. The peacock is not globally destocking because there is no possible value of $b$ that satisfies $b<c(t)$ for every $t$. Nevertheless the behaviour is the same.

\begin{lem}
Let $[\mu]$ be a destocking peacock. With the same notations as above for $a,b,c$ and $(\mu_t^i)_{i\in\{1,2,3\}}$, we assume moreover that
\begin{itemize}
\item the functions $t\mapsto b(t)$ and $t\mapsto c(t)$ are smooth,
\item For every $i=1,2,3$, $t\mapsto \mu^i_t(\R)$ is smooth with non-zero derivative,
\item $t\mapsto \mu^1_t$ is smooth,
\item $t\mapsto\mu^2_t(\R)$ decreases from $1$ to $0$.
\end{itemize}

There is a unique limit curtain process in $\LIM([\mu])$. This is a locally constant process with exactly one jump.
\end{lem}
\begin{proof}
For proving the convergence in the Skorokhod topology, we apply Theorem 12.6 of \cite{Bi}. First we notice that the processes are concentrated on the path $x_{j,k,l}:t\in[0,1]\to \R$ defined by $x_{j,k,l}(t)=k\chi_{[0,j[}(t)+l\chi_{[j,1]}(t)$ where $0<j<1$, $k\in[a,b]$ and $l\notin[a,b]$. In the subspace made of the latter c\`adl\`ag paths, pointwise convergence on a countable, dense subset of $[0,1]$ provides convergence in the Skorokhod topology. 

Therefore according to \cite[Theorem 12.6]{Bi} it is enough to check the finite dimensional convergence of the processes of the sequence to some limit process. The description of the left-curtain coupling between $\mu_s$ and $\mu_t$ together with the assumptions of the lemma, provide a candidate limit process that we describe now: start from a point $k\in[a,b]$ according to $\mu^2_0=\mu^2$, be constant until time $b^{-1}(k)$. At this time start a second constant trajectory, either at point $c\circ b^{-1}(k)$ or in a point uniformly chosen according to $\dd\mu_t^1/\dd t|_{t={b^{-1}(k)}}$ with the proper probabilities making this transition a martingale kernel. Given finitely many times $t_1,\ldots, t_j$ and a partition $\sigma$, it is enough to consider the trajectories that jump outside the intervals containing the times $t_i$. These trajectories can easily be coupled with the trajectories of the candidate limit process. As is the proof of Theorem \ref{Poisson_fini}, this proves that the Prokhorov distance associated to the Skorokhod distance on $D([0,1])$ tends to zero when the mesh $|\sigma|$ tends to zero.
\end{proof}
\subsubsection{Putting the two steps together}
We consider a peacock parametrised on $[0,2]$ that we illustrate on Figure \ref{counter}. When restricted to $[0,1]$, it is simply the peacock of Example \ref{stock}. On $[1,2]$ it is a (locally) destocking peacock. It is made of four terms $\mu_t=\mu^1_t+\mu^2_t+(\mu^3_t+\mu'^3_t)$. We start to define the easy parts
$$\mu'^3_t=\e^{-2}\Lg_{[\e^{2}/2-\e,\,\e^{2}/2]}\AND\mu^2_t(\dd x)=\frac{1}{(-2x)^{3/2}}\Lg_{[-\e^{2}/2,\,-\e^{2(t-1)}/2]}(\dd x)$$
that continue or restrict the two parts of $\mu_t$ after $t=1$. We have also
$$\mu^1_t=f(t,x)\Lg_{[-6,\,-5]}\AND \mu^3(\dd x)=g(x)\lambda_{[1-\e^{2(t-1)}/2,\,1/2]}(\dd x).$$
where $f$ and $g$ are chosen in a way that the mass of $-\dd\mu^2_t=\mu^2_{t+\dd t}-\mu^2_t$ that is at first order $\e^{-3(t-1)}\lambda_{[-\e^{2(t-1)}/2-\e^{2(t-1)}\dd t,-\e^{2(t-1)}/2]}$ is mapped on the atom $\dd\mu^3_t\approx \e^{2(t-1)}\dd t\,g(1-\exp(2(t-1))/2)\delta_{2^{-1}(1-\exp(2(t-1)))}$ for the upper part of the left-curtain coupling or \emph{linearly} on $\dd\mu^1_t\approx (\dd f(t,x)/\dd t)\lambda_{[-6;-5]}\dd t$ for the down part. This is obtained for the functions
$$f(t,x)=\int_1^t\frac{\e^{-(u-1)}}{1-x-\e^{-2(u-1)}/2}\dd u\AND g(x)=[1+\ln(1-(x-6)^{-1})](-2(x-1))^{3/2}.$$

Let $(\sigma^{(p)})$ be a sequence of partitions of $[0,1]$ with mesh $|\sigma^{(p)}|$ going to $0$. We associated any $\sigma^{(p)}$ with the partition of $[0,2]$ made of the times $t$ of $\sigma^{(p)}$ together with the times $t+1\in[1,2]$. We denote the latter by $\bar{\sigma}^{(p)}$. Note that $1$ is a time of this partition. We have seen in the two former paragraphs that when restricted to $[0,1]$ or $[1,2]$ the peacock $[\mu]$ has a unique limit curtain process and one can check that it is Markovian in the two cases. For the peacock on $[0,2]$ and the sequence $(\bar{\sigma}^{(p)})$ one also obtains a limit curtain process but rather surprisingly it is not a Markovian process. 

The proof is technical and not really different from the one of Proposition \ref{pro_uniform}. Instead of providing the details we will explain what happens. If $t_{k-1}$ and $t_k$ are elements of $\sigma^{(p)}$, different trajectories of $X^{\mu,p}$ are jumping down at time $t_k$. They are mapped linearly from $[\exp(2t_{k-1})/2-\exp(t_{k-1}), \exp(2t_{k-1})/2]$ to the small interval $[-\exp(2t_{k-1})/2,-\exp(2t_{k})/2]$. Between $t_k$ and $t_{k-1}+1$ nothing can happen to these trajectories because the left-curtain couplings is identity in their regions. At time $t_k+1$ all the mass contained in the small interval $[-\exp(2t_{k-1})/2,-\exp(2t_{k})/2]$ must jump again either to the neighbourhood of $1-\exp(2t_k)/2$ or somewhere down into the interval $[-6,-5]$. We have parametrised the masses $\mu^1_t$ and $\mu^3_t$ in such a way that the jump down becomes linear when $|t_k-t_{k-1}|\leq |\bar\sigma^{p}|$ is small. Therefore the mapping of the positions between time $t_{k-1}$ and $1+t_k$ is almost linear. More precisely the left-curtain transitions reverses the orientation at the first jump and make it right again at the second. In the limit curtain process, there are two types of trajectories. The first type is made of the continuous trajectories 
\begin{align*}
t\mapsto \min\left(\exp(2t)/2-(1/2-X_0)\exp(t),\exp(2)/2-(1/2-X_0)\exp(1)\right).
\end{align*}
 We are interested in the second type of trajectories that start in the same way but  jump after a duration $T<1$ (an exponential random time). After the jump the trajectory has value $-\exp(2T)/2$ on $[T,T+1[$. There is a second jump at time $T+1$ either to $-\exp(2T)/2+1$ or to a point of $[-6,-5]$ that depends of the past in the simplest manner. Indeed, this point is $X_0-(6-1/2)$. Hence the limit curtain process is not Markovian.

\subsection{Open questions}
We address the following open problems. Given a right-continuous peacock $[\mu]$, we ask

\begin{itemize}
\item Is the set $\LIM([\mu])$ not empty?
\item How many Markov process may $\LIM([\mu])$ contain? We conjecture that there is exactly one.
\end{itemize}

One may address the same questions for $\LFD([\mu])$ or for peacocks without continuity assumptions.
Also the same questions make sense for other couplings than the left-curtain coupling, still using the Markov composition. With respect to the Kremel-Kamae Theorem \cite{KK}, the case of the quantile coupling seems of interest.

\subsection*{Acknowledgements}
I wish to thank Mathias Beiglb\"ock for helpful discussions from the elaboration of our previous article until now. I thank my colleague Vincent Vigon for reading the paper and valuable suggestions as well as Ana Rechtman, Jean B\'erard and Pierre Py for providing references. A lot of thanks go to Harold Gretton for his English and editorial advice. I also wish to thank Xiaolu Tan and Michel \'Emery for making me aware of \cite{HTT} and \cite{HiRoYo} respectively.

\def\cprime{$'$}

\end{document}